\documentclass[english, 12pt,leqno]{amsart}

\usepackage{amsmath, amsfonts, amssymb, enumerate,enumitem}

\newtheorem{thm}{Theorem}[section]
\newtheorem{cor}[thm]{Corollary}
\newtheorem{lem}[thm]{Lemma}
\newtheorem{prop}[thm]{Proposition}

\theoremstyle{definition} 
\newtheorem{defn}[thm]{Definition}

\newtheorem{exe}[thm]{Example}
\newtheorem{rem}[thm]{Remark}

\theoremstyle{remark}

\newcommand{\N}{\mathbf N}
\newcommand{\Z}{\mathbf Z}
\newcommand{\Q}{\mathbf Q}
\newcommand{\R}{\mathbf R}
\newcommand{\C}{\mathbf{C}}
\newcommand{\T}{\mathbf{T}}

\newcommand{\GL}{\textnormal{GL}}
\newcommand{\Sy}{\textnormal{Sym}}
\newcommand{\GF}{\textnormal{GF}}

\newcommand{\ru}{\textnormal{ru}}
\newcommand{\lu}{\ell\textnormal{u}}

\title
[Amenability of  topological groups]
{Amenability and ergodic properties of topological groups:
\\
from Bogolyubov onwards}

\author[R.\ Grigorchuk and P.\ de la Harpe]
{Rostislav Grigorchuk and  Pierre de la Harpe}

\date{July 23, 2015.}

\address{Rostislav Grigorchuk:
Department of Mathematics,
Mailstop 3368,
Texas A{\&}M University,
College Station,
TX 77843--3368, USA.
e-mail: grigorch@math.tamu.edu}

\address{Pierre de la Harpe:
Section de math\'ematiques, 
Universit\'e de Gen\`eve,  
C.P.~64, 
CH--1211 Gen\`eve 4. 
e-mail: Pierre.delaHarpe@unige.ch
}

\thanks{The visit of the first Author to Geneva, where much of this work was done,
was supported by the Swiss National Science Foundation grant 
200020\underline{\phantom{a}}141329.
The first Author acknowledges support of NSF grant DMS-1207699.
We are grateful to
Claire Anantharaman, Ivan Babenko, Dana Bartosova, Bachir Bekka, Yves de Cornulier, 
Julien Melleray, Vladimir Pestov, John Roe, and Klaus Schmidt, 
for useful conversations and mail exchanges, as well as to the referee 
for his constructive criticism of a first version of this article.}


\begin{document}

\begin{abstract}
The purpose of this expository and historical article is to
revisit the notions of amenability and ergodicity, 
and to point out that they appear for topological groups
that are not necessarily locally compact
in articles by Bogolyubov (1939), Fomin (1950), Dixmier (1950),
and Rickert (1967).
\begin{center}
Contents.
\end{center}
\par\noindent
1. Introduction.
\par\noindent
2. Some historical comments
\par\noindent
3. Amenability, fixed point property,  and B-amenability
\par\noindent
4. Hereditary properties
\par\noindent
5. Examples
\par\noindent
6. Extreme amenability
\par\noindent
7. On the definition of ergodicity
\par\noindent
8. Kolmogorov's example
\par\noindent
9. Fomin's representations
\end{abstract}

\maketitle

\section{\textbf{Introduction}}
% s1
\label{Intro}

The notion of amenability for groups has several historical origins.
Most authors cite the article \cite{vNeu--29} of von Neumann,
sometimes complemented by a note \cite{Tars--29} of Tarski,
developed in \cite{Tars--38}.
Von Neumann and Tarski were motivated by
the paradox of Banach and Tarski \cite{BaTa--24},
and its origin in a decisive observation of Hausdorff \cite[Page 469]{Haus--14}.
There is a good and comprehensive exposition 
of this subject in \cite{Wago--85}. 
In these references, groups do not have topology,
in other words they are just ``discrete groups''
(even if topological groups appear in Wagon's book,
they appear only marginally).
Amenability is considered explicitely for topological groups 
in later articles \cite{Dixm--50, Fomi--50, Rick--67}.
But attention has often been restricted to two particular classes of topological groups:
 \emph{discrete} groups (and semi-groups), 
as in \cite[$\S$~18]{HeRo--63},
and \emph{locally compact groups},
as in the influential book of Greenleaf \cite{Gree--69}.
\par

The first goal of the present article is to indicate that the notion of amenability
has at least one other independent origin.
Indeed, it was our surprise to discover that topological groups $G$ 
such that there exist left-invariant means on $\mathcal C^{\textnormal{b}} (G)$
are already the main subject of the 1939 article of Bogolyubov \cite{Bogo--39},
where appears the class of \emph{topological groups}
that we call ``B-amenable'' in Definition \ref{defBam} below;
the ``B'' refers to Bogolyubov.
Here, $\mathcal C^{\textnormal{b}} (G)$ stands for the Banach space of
bounded continuous real-valued functions on $G$.
\par

The second goal of this article is to expose and survey 
various aspects of amenability as they occur 
for topological groups that need not be locally compact.
Results stated (and often proved) below are not original.
\par

\vskip.2cm

\textbf{Plan of the article.}
We add some historical comments in Section~\ref{histcom}.
In Section \ref{Fasa},
we review the basic notions:
amenability and the related fixed point property;
and also B-amenability;
for locally compact groups, the three notions are equivalent.
Section \ref{herprop} is about their hereditary properties.
Section \ref{Sexamples} contains three examples of amenable topological groups
that are not locally compact:
the unitary group of a separable infinite dimensional Hilbert space, 
the symmetric group $\Sy (\N)$ of the positive integers,
and the general linear group $\GL (V)$ of a vector space
of countable infinite dimension over a finite field
(with their Polish topologies).
Section \ref{extreme} provides more examples which have indeed a stronger property:
they are extremely amenable.
In Section \ref{defergetc}, 
we discuss several definitions of ergodicity for actions on compact spaces;
they agree for second countable locally compact groups,
but not in general, as shown in Section \ref{Kolmo} 
by an Example of Kolomogorov involving $\Sy (\N)$.
The final Section \ref{srepresentations} describes 
a characterization, due to Fomin, of ergodicity in terms of unitary representations.

\vskip.2cm

It is convenient to agree that
\textbf{all topological spaces and groups appearing in this article
are assumed to be Hausdorff.}

\section{Some historical comments}
% s2
\label{histcom}

Let $G$ be a topological group and $X$ a non-empty compact space
on which $G$ acts continuously by homeomorphisms.
Denote by $\mathcal P (X)$ the space of probability measures on $X$,
by $\mathcal P ^G (X)$ its subspace of $G$-invariant probability measures,
and by $\mathcal E ^G (X)$ the subspace 
of indecomposable $G$-invariant probability measures on $X$
(the definitions of indecomposable measure, and of related ergodic notions,
are recalled in Section \ref{defergetc}). 
In the classical case of $G$ the additive group $\R$ of real numbers
and $X$ a metrizable compact space,
Bogolyubov and Krylov established in \cite{KrBo--37} three fundamental results,
that we quote essentially as Fomin does in \cite[$\S$~2]{Fomi--50}:
\begin{enumerate}[noitemsep,label=(\arabic*)]
\item\label{1DEintro}
$\mathcal P ^\R (X)$ is a non-empty convex compact subspace 
in the appropriate topological vector space.
\item\label{2DEintro}
$\mathcal P ^\R (X)$ is the convex closure of the space
$\mathcal E ^\R (X)$ of ergodic invariant probability measures;
in particular $\mathcal E ^\R (X)$ is non-empty.
(``Ergodic'' is used today, but ``transitive'' is used in \cite{KrBo--37, Bogo--39, Fomi--50}.)
\item\label{3DEintro}
For every $\mu \in \mathcal E ^\R (X)$, there exists an $\R$-invariant Borel subset
$E_\mu \subset X$ with the following  properties:
$\mu(E_\mu) = 1$, and $\mu'(E_\mu) = 0$
for all $\mu' \in \mathcal E ^\R (X)$, $\mu' \ne \mu$.
\end{enumerate}
See also \cite{Oxto--52} for $G = \Z$ (instead of $\R$).
\par

The generalization from $\R$ to other topological groups raises several problems
and justifies new notions, in particular that of amenable groups,
for which \ref{1DEintro} holds (Proposition \ref{fixedpoint} below).
When \ref{1DEintro} holds, 
\ref{2DEintro} holds without further restriction on $G$,
as it is indicated in \cite{Bogo--39};
this follows alternatively from the Krein-Milman theorem \cite{KrMi--40}.
Result \ref{3DEintro} holds more generally
when $G$ is a second countable locally compact groups,
as exposed in \cite{Vara--63} and \cite[Theorem 1.1(3)]{GrSc--00}, 
see Section \ref{defergetc};
but \ref{3DEintro} does not hold for all groups, 
as shown by the example of Kolmogorov discussed in Section \ref{Kolmo}. 

\vskip.2cm

It is likely that Bogolyubov, and perhaps later also Fomin,
were not aware of von Neumann's and Tarski's 1929 articles 
when they wrote \cite{Bogo--39} and \cite{Fomi--50}.
Moreover, Bogolyubov's 1939 article had very little impact at the time
(see \cite{Anos--94}, in particular the bottom of Page 10).
Indeed, neither \cite{Bogo--39} nor \cite{Fomi--50} appears in the lists of references of
any of \cite{Dixm--50}, \cite{HeRo--63}, \cite{Rick--67}, \cite{Gree--69},  
or \cite{Wago--85}.

Nevertheless, in his 1939 article,
Bogolyubov shows that \ref{1DEintro} above holds for every topological group
which is B-amenable,
and the \emph{same proof} shows that this extends to amenable topological groups.
In $\S$~2 of \cite{Fomi--50}, Fomin shows that the following classes of groups
are B-amenable: compact groups (his Theorem~3),
groups containing a cocompact closed B-amenable subgroup (Theorem 4), 
abelian groups (Theorem 5),
and solvable groups (Corollary 2).
For Fomin's work in general during the 50's, see \cite{Ale+--76}.
\par

A possible other origin of amenability could be an article \cite{Ahlf--35} on Nevanlinna theory,
where Ahlfors defines ``regularly exhaustible'' open Riemann surfaces, i.e.\ surfaces $S$
with a nested sequence $\Omega_1 \subset \Omega_2 \subset \cdots$
of domains with compact closures and smooth boundaries,
such that $S = \bigcup_{n=1}^{\infty} \Omega_n$,
and such that
$\lim_{n \to \infty} 
\frac{ \operatorname{length}_g(\partial \Omega_n)}{ \operatorname{area}_g(\Omega_N)} 
= 0$, 
where the subscript $g$ refers to some metric in the conformal class
defined by the complex structure on $S$.
Ahlfors has used such sequences to define averaging processes,
as F\o lner did later with ``F\o lner sequences'' in groups.
The notion of regular exhaustion has natural formulations in Riemannian geometry;
from several possible references on this subject, we will only quote \cite{Roe--88}.
The connection between Ahlfors' regular exhaustions and amenability seems rather obvious now, 
but we are not aware of any discussion of it in the literature before the 80's.
Apparently, the connection could only be observed after amenability
was recognized as a metric property of both groups and spaces.
\par

Though we will not discuss it here, we note
that the notion of amenability has been extended to many other objects than groups,
including semi-groups, associative algebras, 
Banach algebras, operator algebras (nuclearity, exactness, injectivity), 
metric spaces, equivalence relations, group actions, foliations, and groupoids.

\section{\textbf{Amenability, fixed point property,  and B-amenability}}
% s3
\label{Fasa}

Let $X$ be a topological space.
Let $\mathcal C ^{\textnormal{b}} (X)$ denote the Banach space of 
\textbf{bounded continuous real-valued functions} on $X$,
with the sup norm defined by $\Vert f \Vert = \sup_{x \in X} \vert f(x) \vert$
for $f \in \mathcal C ^{\textnormal{b}} (X)$.
When $X$ is compact, every continuous function on $X$ is bounded,
and we rather write $\mathcal C (X)$ for $\mathcal C ^{\textnormal{b}} (X)$.
\par

Let $G$ be a topological group.
For $a \in G$ and $f \in \mathcal C ^{\textnormal{b}} (G)$, define ${}_a f$ and $f_a$ by
${}_a f (g) = f(a^{-1}g)$ and $f_a (g) = f(g a^{-1})$ for all $g \in G$.
Observe that, for a given $f \in \mathcal C ^{\textnormal{b}} (G)$,
the map $G \longrightarrow \mathcal C ^{\textnormal{b}} (G), \hskip.1cm a \longmapsto {}_a f$ 
need not be continuous
(example with $G = \R$ and $f(t) = \sin (\pi t^2)$ for all $t \in \R$).
Let $\mathcal C ^{\textnormal{b}}_{\ru} (G)$ 
denote the subspace of $\mathcal C ^{\textnormal{b}} (G)$ 
of those functions $f$ for which ${}_a f$ depends continuously of $a$,
i.e.\ the space of\footnote{Here is a justification of the word ``right''.
The \textbf{right-uniform structure} on $G$ is rightfully the uniform structure
with entourages of the form $U_V = \{ (g,h) \in G \times G \mid hg^{-1} \in V \}$,
for some neighbourhood $V$ of $1$ in $G$; this structure is invariant by right multiplications.
Hence a function $f : G \longrightarrow \R$ is \textbf{right-uniformly continuous} if,
for all $\varepsilon > 0$, there exists a neighbourhood $V$ of $1$ in $G$ such that
\begin{equation*}
\aligned
&(g,h) \in U_V \hskip.3cm \Rightarrow \hskip.3cm
    \vert f(h) - f(g) \vert < \varepsilon ,
\\
i.e.\ \hskip.2cm
& \hskip.5cm a \in V \hskip.6cm \Rightarrow \hskip.3cm
    \vert f(ag) - f(g) \vert < \varepsilon \hskip.2cm \forall \hskip.2cm g \in G ,
\\
i.e.\ \hskip.2cm
& \hskip.5cm a \in V \hskip.6cm \Rightarrow \hskip.3cm
    \Vert {}_a f - f \Vert < \varepsilon .
\endaligned
\end{equation*}
This is why we use ``right'' here, as for example in \cite{Rick--67}, 
even though, in the last inequality, $a$ appears on the \emph{left} of $f$;
other authors use ``left'' at the same place \cite[Page 136]{Zimm--84}.}
\textbf{bounded right-uniformly continuous functions} on $G$.
This is a Banach space, indeed a closed subspace 
of the Banach space $\mathcal C^{\textnormal{b}} (G)$.
[We use $a^{-1}$ in the definition of ${}_a f$, 
so that $(a,f) \longmapsto {}_a f$ defines a left-action of $G$ on 
$C ^{\textnormal{b}} (G)$ and $\mathcal C ^{\textnormal{b}}_{\ru} (G)$;
some authors use $a$ rather than $a^{-1}$; 
this does not change the definition of $\mathcal C ^{\textnormal{b}}_{\ru} (G)$.]
Similarly, the space $C^{\textnormal{b}}_{\lu}(G)$ of
\textbf{bounded left-uniformly continuous functions} on $G$,
i.e.\ the space of those $f \in \mathcal C^{\textnormal{b}}(G)$
for which $f_a$ depends continuously on $a$,
is a closed subspace of $ \mathcal C^{\textnormal{b}}(G)$.
We denote by  $\mathcal C^{\textnormal{b}}_u(G)$ the intersection  
$\mathcal C^{\textnormal{b}}_{\ru}(G) \cap  \mathcal C^{\textnormal{b}}_{\lu}(G)$.
\par

Let $E$ be a linear subspace of $\mathcal C ^{\textnormal{b}} (G)$ 
containing the constant functions.
A \textbf{mean} on $E$ is a linear form $M$ on $E$ 
that is \textbf{positive}, i.e.\ $M (f) \ge 0$ whenever $f(g) \ge 0$ for all $g \in G$,
and \textbf{normalized}, i.e.\ $M (1) = 1$, where $1$ denotes both
the number $1 \in \R$ and the corresponding constant function on $G$.
Equivalently, a mean is a linear form on $E$ such that
\begin{equation*}
\inf_{g \in G} f(g) \, \le \, M(f) \ \le \, \sup_{g \in G} f(g)
\hskip.5cm \forall \hskip.2cm f \in E .
\end{equation*}
Observe that a mean $M$ on $\mathcal C ^{\textnormal{b}} (G)$ is bounded of norm $1$.
Assume moreover that $E$ is such that
${}_a f$ is in $E$ for all $f \in E$ and $a \in G$.
A mean $M$ on $E$ is \textbf{left-invariant} if 
$M ({}_a f) = M(f)$ for all $f \in E$ and $a \in G$.
\textbf{Right-invariant} means are defined similarly.
Assume that $E$ is such that 
both ${}_af$ and $f_a$ are in $E$ for all $f \in E$ and $a \in G$;
a mean on $E$ is \textbf{bi-invariant} if it is both left-invariant and right-invariant.

\begin{prop}
% 3.1
\label{left/right}
For a topological group $G$, the following properties are equivalent:
\begin{enumerate}[noitemsep,label=(\roman*)]
\item\label{iDEleft/right}
there exists a left-invariant mean  
on $\mathcal C ^{\textnormal{b}}_{\ru} (G)$,
\item\label{iiDEleft/right}
there exists a right-invariant mean 
on $\mathcal C ^{\textnormal{b}}_{\lu} (G)$,
\end{enumerate}
and they imply
\begin{enumerate}[noitemsep,label=(\roman*)]
\addtocounter{enumi}{2}
\item\label{iiiDEleft/right}
there exists a bi-invariant mean 
on $\mathcal C ^{\textnormal{b}}_{u} (G)$.
\end{enumerate}
\end{prop}

\begin{proof}
For a function $f$ defined on $G$, denote by 
${f}^{\vee}$
the function $g \longmapsto f(g^{-1})$.
Observe that $(f^\vee)^\vee = f$, and that,
for $f \in \mathcal C ^{\textnormal{b}} (G)$,
we have $f \in \mathcal C ^{\textnormal{b}}_{\ru} (G)$
if and only if ${f}^{\vee} \in \mathcal C ^{\textnormal{b}}_{\lu} (G)$.
\par
If there exists a left-invariant mean $M$ on  $C ^{\textnormal{b}}_{\ru} (G)$,
then $f \longmapsto M({f}^\vee)$ is a right-invariant mean on $C ^{\textnormal{b}}_{\lu} (G)$.
Hence \ref{iDEleft/right} implies \ref{iiDEleft/right}. 
Similarly, \ref{iiDEleft/right} implies \ref{iDEleft/right}.
\par
Assume that there exist a left-invariant mean 
$M_\ell$ on  $C ^{\textnormal{b}}_{\ru} (G)$
and a right-invariant mean $M_{\textnormal{r}}$ on  $C ^{\textnormal{b}}_{\lu} (G)$.
For $f \in C ^{\textnormal{b}}_{u} (G)$, define first $F_f : G \longrightarrow \R$
by $F_f(a) = M_\ell(f_a)$; then $F_f \in C ^{\textnormal{b}}_{u} (G)$;
define now a linear form $M$ on $C ^{\textnormal{b}}_{u} (G)$
by $M(f) = M_{\textnormal{r}}(F_f)$. 
Then $M$ is a bi-invariant mean on $C ^{\textnormal{b}}_{u} (G)$.
This shows that \ref{iDEleft/right} or/and \ref{iiDEleft/right} implies \ref{iiiDEleft/right}.
\end{proof}

We do not know whether \ref{iiiDEleft/right} implies \ref{iDEleft/right} and \ref{iiDEleft/right}. 
It does for locally compact groups: 
see ``References for the proof'' after Proposition \ref{LCstrong=fp}.

\begin{defn}
% 3.2
\label{defam}
A topological group $G$ is \textbf{amenable} if it has 
Properties \ref{iDEleft/right} and \ref{iiDEleft/right} of the previous proposition.
\end{defn}

We write LCTVS as a shorthand for
``locally convex topological vector space'', here on the field of real numbers.
Let $C$ be a convex subspace of an LCTVS; a transformation $g$ of $C$
is \textbf{affine} if
\begin{equation*}
g(c x + (1-c)y) \, = \,  c g(x) + (1-c)g(y)
\end{equation*}
for all $x,y \in C$ and $c \in \R$ with $0 \le c \le 1$.
An action of a topological group $G$ on $C$ is \textbf{continuous}
if the corresponding map $G \times C \longrightarrow C$ is continuous,
and \textbf{affine} if $x \longmapsto g(x)$ is affine for all $g \in G$.
Group actions below are \textbf{actions from the left},
unless explicitly written otherwise.

\begin{defn}
% 3.3
\label{defFP}
A topological group $G$ has the \textbf{fixed point property (FP)} if
every continuous affine action of $G$
on a non-empty compact  convex set in an LCTVS
has a fixed point.
\end{defn}

Definition \ref{defFP} plays a fundamental role in a work of Furstenberg \cite{Furs--63}.
\par

It is straightforward to check that compact groups have Property (FP).
For abelian groups, the following fixed-point theorem 
appeared first in \cite{Mark--36} and \cite{Kaku--38};
a convenient reference 
is \cite[Theorem 3.2.1]{Edwa--65}.

\begin{thm}[Markov-Kakutani theorem]
% 3.4
\label{MK}
Every abelian group has Property (FP).
\end{thm}

The result carries over from abelian groups to solvable groups,
by Proposition \ref{heram}\ref{3DEheram} below, 
and from solvable groups to topological solvable groups, obviously.
Therefore, we have also:

\begin{cor}
% 3.5
\label{MKcor}
Every solvable topological group has Property (FP).
\end{cor}

\begin{prop}
% 3.6
\label{fixedpoint}
For a topological group $G$, the following three properties are equivalent:
\begin{enumerate}[noitemsep,label=(\roman*)]
\item\label{iDEfixedpoint}
$G$ has Property (FP);
\item\label{iiDEfixedpoint}
$G$ is amenable;
\item\label{iiiDEfixedpoint}
for every non-empty compact space $X$ and every continuous action of $G$ on~$X$,
there exists a $G$-invariant probability measure on $X$.
\end{enumerate}
\end{prop}

Before the proof, we recall two standard facts from functional analysis.
\par

\emph{Unit balls in duals of Banach spaces.}
For a real Banach space $E$, we denote by $E^*$ its Banach space dual,
and $E^*_1$ the unit ball in $E^*$. 
On $E^*$ and $E^*_1$, we consider also the weak-$*$-topology, 
for which $E^*$ is a LCTVS
and $E^*_1$ a compact space (this is the Banach-Alaoglu theorem).
\par

\emph{Barycentres.}
Let $C$ be a non-empty compact convex subspace of an LCTVS $E$,
and let $\mu$ be a probability measure on $C$.
Then there exists a unique point $b_\mu \in E$ such that
$f(b_\mu) = \int_C f(x) d\mu (x)$ 
for every continuous linear form $f$ on $E$;
moreover $b_\mu \in C$.
The point $b_\mu$ is called the \textbf{barycentre} of $\mu$,
or the \emph{resultant} of $\mu$, and $\mu$ \emph{represents} $b_\mu$.
Recall the following formulation of the Krein-Milman theorem:
every point of $C$ is the barycentre of a probability measure on $C$
which is supported by the closure of the set of extreme points of $C$;
see \cite[Section 1]{Phel--66}.

\begin{proof}
\ref{iDEfixedpoint} $\Rightarrow$ \ref{iiDEfixedpoint}
The set $\operatorname{Mean}(G)$ of all means on $\mathcal C ^{\textnormal{b}}_{\ru} (G)$ 
is a subspace of the unit ball $\mathcal C ^{\textnormal{b}}_{\ru} (G)^*_1$,
and is closed for the weak-$*$-topology,
so that $\operatorname{Mean}(G)$ is a compact convex set.
Moreover the natural action 
of $G$ on $\operatorname{Mean}(G)$ is affine and continuous
(in the sense recalled just before Definition \ref{defFP}).
\par
If $G$ has Property \ref{iDEfixedpoint}, there exists a $G$-invariant probability measure 
$\mu$ on $\operatorname{Mean}(G)$.
Such a measure has a barycentre $M \in \operatorname{Mean}(G)$.
Since $M$ is $G$-invariant (by uniqueness of the barycentre), 
$G$ has Property \ref{iiDEfixedpoint}.
\par

[It is even more straightforward to check that \ref{iDEfixedpoint} implies \ref{iiiDEfixedpoint},
because if $G$ acts on $X$, 
then $G$ acts on the space $\mathcal P (X)$ of probability measures on $X$,
that is naturally a compact convex set.]

\vskip.2cm

\ref{iiDEfixedpoint} $\Rightarrow$ \ref{iiiDEfixedpoint}
We reformulate the argument of \cite{Bogo--39}, 
written there for $\mathcal C ^{\textnormal{b}} (G)$;
the \emph{same argument} applies equally well to $\mathcal C ^{\textnormal{b}}_{\ru} (G)$,
as can be read for example in \cite{Rick--67}.
\par

Let $\nu$ be a probability measure on $X$.
For $f \in \mathcal C (X)$, define a function
\begin{equation*}
F_f \, : \, G  \longrightarrow \R , \hskip.2cm g \longmapsto \int_X f(gx) d\nu(x) .
\end{equation*}
Then $F_f$ is bounded and right-uniformly continuous on $G$.
Let moreover $a \in G$. 
Then
\begin{equation*}
F_{({}_a f)} (g) \, = \, 
\int_X f(a^{-1}gx) d\nu(x) \, = \,  
F_f (a^{-1}g) \, = \,  
\big( {}_{a} (F_f) \big) (g)
\end{equation*}
for all $g \in G$. 
\par

Let $M$ be left-invariant mean on $\mathcal C ^{\textnormal{b}}_{\ru} (G)$.
Define a linear form
\begin{equation*}
\mu \, : \, \mathcal C (X)  \longrightarrow \R , \hskip.2cm f \longmapsto M(F_f) .
\end{equation*}
Then $\mu$ is a normalized positive linear form on $\mathcal C (X)$,
i.e.\ $\mu$ can be seen as a probability measure on $X$.
Since $M$ is left-invariant, we have 
\begin{equation*}
\mu({}_af) \, = \, M\big( F_{({}_af)} \big) \, = \, M\big({}_a(F_f) \big) \, = \,
M(F_f) = \mu(f)
\end{equation*}
for all $f \in \mathcal C (X)$ and $a \in G$, 
i.e.\ the measure $\mu$ is  $G$-invariant.

\vskip.2cm

\ref{iiiDEfixedpoint} $\Rightarrow$ \ref{iDEfixedpoint}
Let $C$ be a compact convex set on which $G$ acts continuously, by affine transformations.
If $G$ has Property (iii), there exists a $G$-invariant probability measure $\mu$ on $C$.
The barycentre $b_\mu \in C$ of $\mu$ is fixed by $G$.
\end{proof}

\begin{rem}
% 3.7
\label{remsur??}
(1)
The equivalence of \ref{iDEfixedpoint} and \ref{iiDEfixedpoint} 
appears in many places, for example in \cite{Eyma--75},
who quotes \cite{Day--61} 
(as much as we know the first article where the arguments are found,
but applied to abstract groups only)
and Rickert 
(as much as we know the first author who applies the arguments
to topological groups); see more precisely \cite[Theorem 4.2]{Rick--67}.

\vskip.2cm

(2)
For a direct proof of \ref{iiDEfixedpoint} $\Rightarrow$ \ref{iDEfixedpoint}, 
see also \cite[Theorem G.1.7]{BeHV--08}.

\vskip.2cm

(3)
There is a natural embedding 
$G \longrightarrow \mathcal C ^{\textnormal{b}}_{\ru} (G)^*_1$,
$g \longmapsto (F \mapsto F(g))$,
which is continuous for the weak-$*$-topology on the range
(recall that the subscript $1$ stands for ``unit ball'').
By definition, the \textbf{universal equivariant compactification} $\gamma_u G$ of $G$
is the closure of the image of this embedding,
and the natural action of $G$ on $\gamma_u G$ is continuous.
The properties of Proposition \ref{fixedpoint}
are moreover equivalent to  
\begin{enumerate}[noitemsep,label=(\roman*)]
\addtocounter{enumi}{3}
\item\label{ivDEleft/right}
the natural action of $G$ on its universal equivariant compactification $\gamma_u G$
has an invariant probability measure,
\end{enumerate}
as it is shown in \cite{BaBo--11}.

\vskip.2cm

(4)
Finally, Properties \ref{iDEleft/right} to \ref{ivDEleft/right} are equivalent to
\begin{enumerate}[noitemsep,label=(\roman*)]
\addtocounter{enumi}{4}
\item\label{vDEleft/right}
every continuous action of $G$ on the Hilbert cube has an invariant probability measure,
\end{enumerate}
as has been established in \cite{AlMP--11},
and previously in \cite{BoFe--07} for countable groups.
For a countable group $\Gamma$, amenability is also equivalent to
the following property \cite{GiHa--97} :
\begin{enumerate}[noitemsep,label=(\roman*)]
\addtocounter{enumi}{5}
\item\label{viDEleft/right}
every action  by homeomorphisms of $\Gamma$ on the Cantor middle-third space
has an invariant probability measure.
\end{enumerate}
\end{rem}

\vskip.2cm

The end of this section is devoted to a notion stronger than amenability.
It was introduced by Bogolyubov \cite{Bogo--39}, without a name,
and appears in \cite[Definition 2.1]{Rick--67}, unfortunately with a name
used most often later for the notion of our Definition \ref{defam};
we call it B-amenability (Definition \ref{defBam}).
Historically, B-amenability came before amenability for topological groups.
Amenability, equivalent to Property (FP) of Definition \ref{defFP},
has been so far more often useful
than B-amenability, and deserves the shorter name.
(See also Remark \ref{remterminology}.)

\begin{defn}
% 3.8
\label{defBam}
A topological group $G$ is \textbf{B-amenable} if
there exists a left-invariant mean on $\mathcal C^{\textnormal{b}} (G)$.
\end{defn}

\begin{prop}
% 3.9
A B-amenable topological group is amenable.
\end{prop}

\begin{proof}
If $G$ is a topological group, 
the space $\mathcal C ^{\textnormal{b}}_{\ru} (G)$
is a $G$-invariant subspace of $\mathcal C ^{\textnormal{b}} (G)$,
and the restriction to this subspace of a $G$-invariant mean on $\mathcal C ^{\textnormal{b}} (G)$
is a $G$-invariant mean on $\mathcal C ^{\textnormal{b}}_{\ru} (G)$.
\end{proof}

Propositions \ref{LCstrong=fp} and \ref{Ustrong}
establish that the converse holds for locally compact groups,
but not in general. 

\begin{prop}
% 3.10
\label{LCstrong=fp}
Let $G$ be a locally compact group.
Then $G$ is B-amenable if and only if $G$ is amenable.
\end{prop}

\begin{proof}[Reference for the proof]
See  \cite[Theorem 2.2.1 Page 26, and Page 29]{Gree--69}.
Greenleaf shows there that the following properties
of a locally compact group $G$ are equivalent: 
existence of a left-invariant mean on each of 
\begin{enumerate}[noitemsep,label=(\roman*)]
\item\label{iDELCstrong=fp}
the space $L^\infty (G)$ of essentially bounded Borel measurable functions
(modulo equality in complements of locally null sets),
\item\label{iiDELCstrong=fp}
the space $\mathcal C^{\textnormal{b}} (G)$,
\item\label{iiiDELCstrong=fp}
the space $\mathcal C ^{\textnormal{b}}_{\ru} (G)$,
\item\label{ivDELCstrong=fp}
the space $\mathcal C ^{\textnormal{b}}_{u} (G)$.
\end{enumerate}
Moreover, Properties \ref{iDELCstrong=fp} to \ref{ivDELCstrong=fp} 
are equivalent to
\begin{enumerate}
\item[(*)]
there exists a bi-invariant mean on $E$,
\end{enumerate}
for $E$ any one of the spaces in \ref{iDELCstrong=fp} to \ref{ivDELCstrong=fp} above.

Indeed, a substantial part of the early theory of amenability on locally compact groups
is to show that infinitely many definitions are equivalent with each other.
\end{proof}

\begin{exe}
% 3.11
\label{compactgroups}
Compact groups are B-amenable, since
the normalized Haar measure on a compact group $G$
provides a mean on $\mathcal C^{\textnormal{b}} (G)$
that is both left- and right-invariant.
\end{exe}

\begin{prop}
% 3.12
\label{DixStrAm}
For a topological group $G$, the following two properties are equivalent:
\begin{enumerate}[noitemsep,label=(\roman*)]
\item\label{iDEDixStrAm}
$G$ is B-amenable,
\item\label{iiDEDixStrAm}
for all $n \ge 1$,  $f^{(1)}, \hdots, f^{(n)} \in \mathcal C^{\textnormal{b}} (G)$,  
$a_1, \hdots, a_n \in G$, and $t \in \R$,
such that $f^{(1)} - {}_{a_1} f^{(1)} + \cdots + f^{(n)} - {}_{a_n} f^{(n)} \ge t$,
we have $t \le 0$.
\end{enumerate}
\end{prop}

\begin{proof}[On the proof]
The non-trivial implication is \ref{iiDEDixStrAm} $\Rightarrow$ \ref{iDEDixStrAm}.
It is a consequence of the Hahn-Banach theorem;
see \cite[Th\'eor\`eme 1]{Dixm--50}.
\end{proof}

This has the following consequence, for which we refer 
again to Dixmier
\cite[Th\'eor\`eme 2($\alpha$)]{Dixm--50}:

\begin{prop}
% 3.13
\label{amStrAm}
Abelian topological groups are B-amenable.
\end{prop}

By Proposition \ref{hersam},
we have also the following corollary (compare with \ref{MKcor}):

\begin{cor}
% 3.14
\label{solStrAm}
Solvable topological groups are B-amenable.
\end{cor}

\begin{rem}[on terminology]
% 3.15
\label{remterminology}
In a first version of the present article, we used ``strongly amenable'' for ``B-amenable''.
But this was unfortunate, because the terminology is already used for groups
of which the universal minimal proximal flow is trivial,
and this is a different notion, since there is an example of Furstenberg
which shows that a solvable group need not be strongly amenable
\cite[Page 28]{Glas--76}.
Strong amenability plays its role in more recent work \cite[Section 4]{MeNT}.
% \cite{Day--57} encore un autre sens.
\par

The terminology ``u-amenable'' and ``amenable'' is used in \cite{Harp--73},
instead of ``amenable'' and ``B-amenable'' here.
\end{rem}

\section{\textbf{Hereditary properties}}
% s4
\label{herprop}

In most of this section, we address topological groups in general.
However, groups are assumed to be locally compact in Proposition \ref{closedinLC},
and metrizable in Corollary \ref{coverings}.
\par

A topological group $G$ is the \textbf{directed union} 
of a family $(H_\alpha)_{\alpha \in A}$ of closed subgroups of $G$
if the following conditions hold:
(1) $G = \bigcup_{\alpha \in A} H_\alpha$;
(2) for every $\alpha, \beta \in A$, there exists $\gamma \in A$
such that $H_\alpha \cup H_\beta \subset H_\gamma$;
(3) $G$ has the topology of the inductive limit of the $H_\alpha$ 's.
[Note that the index set $A$ is a directed set for the preorder
defined by $\alpha \le \beta$ if $H_\alpha \subset H_\beta$.]

\begin{prop}[on amenability]
% 4.1
\label{heram}
Let $G$ be a topological group.
\begin{enumerate}[noitemsep,label=(\arabic*)]
\item\label{1DEheram}
If $G$ is amenable, then every open subgroup of $G$ is amenable.
\item\label{2DEheram}
If $G$ is a directed union of a family $(H_\alpha)_{\alpha \in A}$ of closed subgroups,
and if each $H_\alpha$ is amenable, then $G$ is amenable.
\item\label{3DEheram}
If $G$ has an amenable closed normal subgroup $N$ such that the quotient $G/N$ is amenable,
then $G$ is amenable.
\item\label{4DEheram}
Let $H$ be a topological group such that there exists a continuous homomorphism
$H \longrightarrow G$ with dense image;
if $H$ is amenable, then so is $G$.
\item\label{5DEheram}
If $H$ is a dense subgroup of $G$, then $G$ is amenable if and only if $H$,
endowed with the induced topology, is amenable.
\end{enumerate}
\end{prop}

\begin{proof}[On the proof]
For each claim, there is at least one property of Proposition \ref{fixedpoint}
which makes the proof rather straightforward.
As a sample, we check the last claim, 
and refer to \cite[Section 4]{Rick--67} for \ref{1DEheram} to \ref{4DEheram}.
\par

\ref{5DEheram}
Since right-uniformly continuous functions can be extended from $H$ to $G$,
the restriction of functions provides an isomorphism of Banach spaces
$\mathcal C ^{\textnormal{b}}_{\ru} (G) \longrightarrow \mathcal C ^{\textnormal{b}}_{\ru} (H)$,
by which these spaces can be identified; 
let us denote it (them) by $\mathcal C ^{\textnormal{b}}_{\ru}$.
\par
On the one hand, 
a left-$G$-invariant mean on $\mathcal C ^{\textnormal{b}}_{\ru}$
is obviously a left-$H$-invariant mean on $\mathcal C ^{\textnormal{b}}_{\ru}$.
On the other hand, 
since the action of $G$ on $\mathcal C ^{\textnormal{b}}_{\ru}$ is continuous, 
a left-$H$-invariant mean on this space is also left-$G$-invariant.
Claim (5) follows.
\end{proof}

\begin{rem}
% 4.2
\label{remonheram}
(1)
In the first claim of Proposition \ref{heram}, 
``open'' cannot be replaced by ``closed'' (see Corollary \ref{closedsubg}).
See however Propositions \ref{closedinLC} and \ref{fibration}.

\vskip.2cm

(2)
Let $G$ be a group with two Hausdorff topologies $\mathcal T_s, \mathcal T_w$
such that $G_s := (G, \mathcal T_s)$ and $G_w := (G, \mathcal T_w)$ are topological groups,
and $\mathcal T_s$ stronger than $\mathcal T_w$.
If $G_s$ is amenable, then $G_w$ is amenable,
as it follows from Claim (4) applied to the continuous identity homomorphism
$\operatorname{id} : G_s \longrightarrow G_w$.
\par
Suppose for example that $\mathcal T_s$ is the discrete topology
and that $G_s$ is amenable; this is the case if $G$ is abelian, or more generally solvable,
by Theorem \ref{MK}.
Then $G_w$ is amenable for every topology $\mathcal T_w$ making $G$ a topological group.
\vskip.2cm

(3)
The proof of Claim (5) \emph{cannot} be adapted to Proposition \ref{hersam}.
Indeed the analogue of Claim (5) \emph{does not} hold for B-amenability
(Remark \ref{remonU(H)}(1)).

\vskip.2cm

(4) 
The following result of Calvin C.\ Moore (1979) is reminiscent of Claim (5),
but the proof uses completely different notions.
Let $G$ be the group of real points of an $\R$-algebraic group,
$H$ an amenable group, and $\varphi : H \longrightarrow G$ a continuous homomorphism;
if $H$ is amenable and $\varphi (H)$ Zariski-dense in $G$, then $G$ is amenable.
We refer to \cite[Theorem 4.1.15]{Zimm--84}.
\end{rem}

\begin{prop}[on B-amenability]
% 4.3
\label{hersam}
Let $G$ be a topological group.
\begin{enumerate}[noitemsep,label=(\arabic*)]
\item\label{1DEhersam}
If $G$ is B-amenable, then every open subgroup of $G$ is B-amenable.
\item\label{2DEhersam}
If $G$ is a directed union of a family $(H_\alpha)_{\alpha \in A}$ of closed subgroups,
and if each $H_\alpha$ is B-amenable, then $G$ is B-amenable.
\item\label{3DEhersam}
If $G$ has a B-amenable closed normal subgroup $N$ 
such that the quotient $G/N$ is B-amenable,
then $G$ is B-amenable.
\item\label{4DEhersam}
Let $H$ be a topological group such that there exist a continuous epimorphism
$H \twoheadrightarrow G$;
if $H$ is B-amenable, then so is $G$.
\end{enumerate}
\end{prop}

\begin{proof}[Reference for the proof]
As for Proposition \ref{heram}, proofs are straightforward.
In \cite{Rick--67}, see respectively Theorems 3.2, 2.4, 2.6, and 2.2.
There are related results 
in the older article by Dixmier \cite{Dixm--50}.
\end{proof}

Amenability of closed subgroups of amenable groups is more subtle.
Let us first recall the following result in the classical setting of locally compact groups.

\begin{prop}
% 4.4
\label{closedinLC}
Let $G$ be a locally compact group and $H$ a closed subgroup.
If $G$ is amenable, then so is $H$.
\end{prop}

\begin{proof}[On the proof]
We mention here two proofs of this statement.
\par
The proof of \cite[Proposition 4.2.20]{Zimm--84}
uses induction from $H$ to $G$ for actions on compact convex sets.
Since Haar measure is an essential ingredient of induction,
this cannot be used for groups that are not locally compact.
\par
The proof of \cite[Theorem 2.3.2]{Gree--69} has three steps.
We denote by $H \backslash G$ the space of $H$-cosets of the form $Hg$;
we could use $G/H$ instead, but this would impose on us right-invariant means
on $\mathcal C^{\textnormal{b}}(\cdot)$, rather than left-invariant means.
\par
For the first step, $G$ is assumed to be second countable.
Then there exists a Borel transversal $T$ for $H \backslash G$,
so that the multiplication map $H \times T \longrightarrow G$, $(h,t) \longmapsto ht$,
is a Borel isomorphism.
This can be used first to extend functions on $H$ to functions on $G$,
and then to show that amenability is inherited from $G$ to $H$
(we refer to Greenleaf's book for details).
For the second step, $G$ is assumed to be $\sigma$-compact.
Then $G$ has a compact normal subgroup $N$ such that $G/N$
is second countable (Kakutani-Kodaira theorem \cite{KaKo--44}).
The first step and Proposition \ref{heram}\ref{3DEheram}
imply again that amenability is inherited from $G$ to $H$.
The final step makes use of the fact that every locally compact group
contains an open $\sigma$-compact subgroup.
\par
In Rickert's article, the proposition is proved 
with additional hypothesis only,
essentially that the quotient of $G$ by its connected component is a compact group
\cite[Section 7]{Rick--67}.
\end{proof}

Corollary \ref{closedsubg} below shows that Proposition \ref{closedinLC}
\emph{does not} carry over to arbitrary topological groups.
In particular this proposition is unlikely to have a completely straightforward proof.
\par

However, with appropriate extra hypothesis, 
B-amenability is inherited by closed subgroups.
This is shown by the following proposition, which is Theorem 3.4 in \cite{Rick--67}.

\begin{prop}
% 4.5
\label{fibration}
Let $G$ be a topological group, and $H$ a closed subgroup.
Assume that $H \backslash G$ is paracompact
and the fibration $\pi : G \longrightarrow H \backslash G$ is locally trivial.
\par
If $G$ is B-amenable, then so is $H$.
\end{prop}

\begin{proof}
Since $G \longrightarrow H \backslash G$ is locally trivial, 
there exist an open cover $(U_\iota)_{\iota \in I}$ of $H \backslash G$
and a family of continuous sections 
$(\sigma_\iota : U_\iota \longrightarrow G)_{\iota \in I}$
such that $\pi \sigma_\iota (x) = x$ for all $\iota \in I$ and $x \in U_\iota$.
Since $H \backslash G$ is paracompact,
there is a partition of unity $(\varphi_\iota)_{\iota \in I}$
subordinate to $(U_\iota)_{\iota \in I}$.
For $f \in \mathcal C^{\textnormal{b}}  (H)$, define $F_f : G \longrightarrow \R$ by
\begin{equation*}
F_f (g) \, = \, \sum_{\iota \in I} \varphi_\iota(\pi(g)) \hskip.1cm
f \left( g (\sigma \pi(g))^{-1} \right) .
\end{equation*}
The function $F_f$ is well-defined,
continuous because the family of the supports
of the $\varphi_\iota$ 's is locally finite,
and obviously bounded. 
The assignment $f \longmapsto F_f$ is a linear map
from $\mathcal C^{\textnormal{b}} (H)$ to $\mathcal C^{\textnormal{b}} (G)$, 
that respects positivity and constant functions.
Moreover, for
$f \in \mathcal C^{\textnormal{b}} (H)$ and $a \in H$,  we have
\begin{equation*}
{}_a(F_f) \, = \,  F_{ ({}_a f) }.
\end{equation*}
Indeed
\begin{equation*}
\aligned
\left( {}_a(F_f) \right)(g) 
\, &= \,
\sum_{\iota \in I} \varphi_\iota ( \pi(a^{-1}g)) \hskip.1cm f \left( a^{-1}g (\sigma\pi (a^{-1}g))^{-1} \right) 
\\
\, &= \,
\sum_{\iota \in I} \varphi_\iota ( \pi(g)) \hskip.1cm f \left( a^{-1}g (\sigma\pi(a^{-1}g))^{-1} \right) 
\, = \,
\left( F_{ {}_a f } \right) (g) 
\endaligned 
\end{equation*}
for all $g \in G$.

Suppose that there exists a left-$G$-invariant mean $M$ on $\mathcal C ^{\textnormal{b}}(G)$.
The assignment $m : f \longmapsto M(F_f)$ 
is obviously a mean on $\mathcal C^{\textnormal{b}}(H)$.
Since $M$ is invariant, we have  
\begin{equation*}
m( {}_a f) \, = \, M( F_{ ({}_a f) } ) \, = \, M({}_a(F_f)) \, = \, M( F_f ) \, = \, m(f)
\end{equation*}
for all $f \in \mathcal C^{\textnormal{b}} (H)$ and $a \in H$, 
i.e.\ $m$ is a left-$H$-invariant mean on $\mathcal C^{\textnormal{b}} (H)$.
\end{proof}

\begin{cor}
% 4.6
\label{coverings}
Let $G$ be a metrizable topological group.
If $G$ contains a non-amenable discrete subgroup $\Gamma$,
then $G$ is not B-amenable.
\end{cor}

\begin{proof}
Since $G$ is metrizable, so is $\Gamma \backslash G$, 
hence $\Gamma \backslash G$ is paracompact.
\par

Let $p : G \longrightarrow \Gamma \backslash G$ denote the canonical projection.
Let $V$ be a neighbourhood of $1$ in $G$ 
such that $V^{-1} = V$ and $\Gamma \cap V^2 = \{1\}$.
For every $g \in G$, the open subsets $\gamma V g$, for $\gamma \in \Gamma$,
are pairwise disjoint; it follows 
that $p^{-1}\big( p(Vg) \big)$ is homeomorphic to the product $\Gamma \times Vg$,
and therefore that the fibration $p$ is locally trivial.
\par

Hence $G$ is not B-amenable by Proposition \ref{fibration}.
\end{proof}

Concerning the hypothesis of Corollary \ref{coverings},
let us recall a theorem due to  Birkhoff and Kakutani;
for a topological group $G$, the three following conditions are equivalent:
\begin{enumerate}[noitemsep,label=(\roman*)]
\item
$G$ is metrizable as a topological space;
\item 
the topology of $G$ can be defined by a left-invariant metric;
\item
$G$ has a countable basis of neighbourhoods of~$1$.
\end{enumerate}
See for example \cite[$\S$~3, no 1, Pages IX.23-24]{BTG5-10}, or \cite{CoHa}.

\begin{exe}
% 4.7
\label{Karube}
In Proposition \ref{fibration},
the property of local triviality is not automatic,
as the following example, from \cite{Karu--58}, shows.
\par
Consider the circle group $T = \R / \Z$, 
its subgroup $C$ of order $2$,
the compact group $G = T^{\N}$,
and the closed subgroup $H = C^{\N}$ of $G$.
Then $G$ is locally connected, as a product of (locally) connected groups,
while $H$ and $H \times (H \backslash G)$ are not locally connected
(indeed $H$ is totally disconnected).
It follows that $G$ and $H \times (H \backslash G)$ are not locally homeomorphic,
and in particular that the projection $G \longrightarrow H \backslash G$ is not locally trivial.
\end{exe}

\section{\textbf{Examples}}
% s5
\label{Sexamples}

Propositions \ref{Ustrong}, \ref{Sym}, and \ref{GL(V)} 
provide examples of topological groups
that are amenable and are not B-amenable.
By necessity, these groups are not locally compact 
(Proposition \ref{LCstrong=fp}).

\vskip.2cm

Let $\mathcal H$ be an infinite dimensional complex Hilbert space.
We denote by $\operatorname{U} (\mathcal H)_{\text{str}}$ it \emph{unitary group},
endowed with the strong topology (equivalently: with the weak topology);
as is well-known, this is a topological group, and it is not locally compact.
For the strong and weak topologies on sets of operators, and their properties,
see e.g.\ \cite{Dixm--57}.
If $\mathcal H$ is separable, $\operatorname{U} (\mathcal H)_{\text{str}}$ is a Polish group,
i.e.\ it is separable and its topology can be defined by a complete metric.
As a curiosity, we note that, for a separable Hilbert space $\mathcal H$, 
the strong topology is the unique topology
for which $\operatorname{U} (\mathcal H)$ is a Polish group \cite{AtKa--12}.

\begin{lem}
% 5.1
\label{lambda}
Let $\mathcal H$ be an infinite dimensional separable complex Hilbert space,
and $\operatorname{U} (\mathcal H)_{\text{str}}$ its unitary group.
\par
Every countable group $\Gamma$ is isomorphic to a discrete subgroup of
$\operatorname{U} (\mathcal H)_{\text{str}}$.
\end{lem}

\begin{proof}
Suppose first that $\Gamma$ is an infinite group.
We can identify $\mathcal H$ with the Hilbert space $\ell^2 (\Gamma)$
of complex-valued functions $\xi$ on $\Gamma$
such that $\sum_{\gamma \in \Gamma} \vert \xi (\gamma) \vert^2 < \infty$.
Let
\begin{equation*} 
\lambda : \Gamma \longrightarrow \operatorname{U}(\ell^2(\Gamma))_{\text{str}}
\end{equation*}
be the left-regular representation of $\Gamma$,
defined by $(\lambda (\gamma) \xi)(\gamma') = \xi(\gamma^{-1} \gamma')$
for all $\gamma, \gamma' \in \Gamma$ and $\xi \in \ell^2 (\Gamma)$.
\par

For $\gamma \in \Gamma$, let $\delta_\gamma \in \ell^2(\Gamma)$
denote the unit vector defined by $\delta_\gamma(\gamma) = 1$
and $\delta_\gamma(\gamma') = 0$ if $\gamma' \ne \gamma$;
observe that $\lambda(\gamma)\delta_1 = \delta_\gamma$.
Set
\begin{equation*}
U_\gamma \, = \, \left\{ g \in \operatorname{U}(\ell^2(\Gamma))_{\text{str}}
\hskip.2cm \big\vert \hskip.2cm
\Vert (g - \lambda(\gamma) )\delta_1 \Vert < \sqrt 2 / 2 \right\} .
\end{equation*}
On the one hand, $U_\gamma$ is open in $\operatorname{U}(\ell^2(\Gamma))_{\text{str}}$
and $\lambda(\gamma) \in U_\gamma$.
On the other hand, since $\Vert \delta_\gamma - \delta_{\gamma'} \Vert = \sqrt 2$
for $\gamma, \gamma' \in \Gamma$, $\gamma \ne \gamma'$, 
the $U_\gamma$~'s are pairwise disjoint.
Hence $\lambda(\Gamma)$ is a discrete subgroup of
$\operatorname{U}(\ell^2(\Gamma))_{\text{str}}$.
\par
If $\Gamma$ is a finite group, we can apply the previous argument to
the direct product $\Gamma \times \Z$, and use the embedding
$\Gamma \simeq \Gamma \times \{0\} \subset \Z$.
\end{proof}

\noindent \emph{Note.}
More generally, for every locally compact group $G$,
the left-regular representation $G \longrightarrow \mathcal U (L^2(G))_{\text{str}}$
is both a continuous homomorphism and
a homeomorphism of $G$ onto a closed subgroup of $\mathcal U (L^2(G))_{\text{str}}$
\cite[Exercise G.6.4]{BeHV--08}.
\par

In general, $L^2(G)$ need not be separable.
It is when $G$ is separable.

\begin{prop}[\cite{Harp--73}]
% 5.2
\label{Ustrong}
Let $\mathcal H$ be an infinite dimensional separable complex Hilbert space.
The topological group  $\operatorname{U} (\mathcal H)_{\text{str}}$
is amenable and is not B-amenable.
\end{prop}

\emph{Remark.} Proposition \ref{ExExAm} establishes 
that $\operatorname{U} (\mathcal H)_{\text{str}}$
has a property much stronger than amenability.

\begin{proof}
Let $(e_n)_{n \ge 1}$ be an orthonormal basis of $\mathcal H$.
For each $n \ge 1$, we identify the compact Lie group $\operatorname{U}(n)$
to the subgroup of $\operatorname{U} (\mathcal H)_{\text{str}}$ of those unitary operators
leaving invariant the linear span $V_n$ of $\{ e_1, \hdots, e_n\}$
and coinciding with the identity on the orthogonal complement $V_n^\perp$.
Let $\operatorname{U} (\infty)$ denote the union of the compact groups in the nested sequence
\begin{equation*}
\operatorname{U} (1) \subset \cdots \subset \operatorname{U}(n)
\subset \operatorname{U}(n+1) \subset \cdots ,
\end{equation*}
with the inductive limit topology.
The group $\operatorname{U} (\infty)$ is dense in $\operatorname{U} (\mathcal H)_{\text{str}}$.
\par
The later claim follows from Kaplansky's density theorem,
essentially the version of \cite[Theorem 2.3.3]{Pede--79}.
Pedersen formulates his Theorem 2.3.3 in terms of
a C$^*$-subalgebra $A$ of $\operatorname{L}(\mathcal H)$.
More generally, his proof applies without change 
to an involutive subalgebra $A$ of $\operatorname{L}(\mathcal H)$
such that $\exp (x) \in A$ for all $x \in A$.
In our case, $A$ is the algebra spanned by the identity operator
and $\bigcup_{n \ge 1} \operatorname{M}_n$, where
$\operatorname{M}_n$ stands for the
the finite-dimensional algebra
$\operatorname{M}_n = \{ x \in \operatorname{L}(\mathcal H) \mid
x(V_n) \subset V_n \hskip.2cm \text{and} \hskip.2cm x(V_n^\perp) = \{0\} \}$.
\par

For the inductive limit topology, 
the group $\operatorname{U} (\infty)$ is a topological group.
It is not locally compact; indeed, it is not a Baire group,
because $\operatorname{U} (n)$ has empty interior in $\operatorname{U} (\infty)$
for all $n$. (For other topological properties of $\operatorname{U} (\infty)$,
see e.g.\ \cite[Theorem 4.8]{Hans--71}).
Since the compact groups $\operatorname{U} (n)$ are amenable, 
so is $\operatorname{U} (\infty)$ by Proposition \ref{heram}\ref{3DEheram}.
It follows from Proposition \ref{heram}\ref{5DEheram}
that $\operatorname{U} (\mathcal H)_{\text{str}}$ is amenable.
\par
By Lemma \ref{lambda}, the group $\operatorname{U} (\mathcal H)_{\text{str}}$
contains non-amenable discrete subgroups.
It follows from Corollary \ref{coverings}
that $\operatorname{U} (\mathcal H)_{\text{str}}$ is not B-amenable.

\vskip.2cm

The proof of the proposition is complete.
Let us however 
reproduce the argument of \cite{Harp--73},
that is a different proof that $\operatorname{U} (\mathcal H)_{\text{str}}$
is not B-amenable.
\par

Suppose ab absurdo that there exists a left-invariant mean $M$ on the space
$\mathcal C^{\textnormal{b}} (\operatorname{U} (\mathcal H)_{\text{str}})$.
Let $\xi$ be a unit vector in $\mathcal H$.
For every bounded operator $S$ on $\mathcal H$,
the function
\begin{equation*}
f_{S,\xi} \, : \, \left\{
\aligned
\operatorname{U}(\mathcal H)_{\text{str}} \hskip.2cm &\longrightarrow \hskip1.5cm \R
\\
g \hskip1cm &\longmapsto \hskip.2cm 
\operatorname{Re}\left( \langle g^{-1} Sg \xi \mid \xi \rangle \right)
\endaligned
\right.
\end{equation*}
is bounded and continuous.
Let $\operatorname{L}(\mathcal H)$ denote the algebra of all
bounded operators on $\mathcal H$.
Observe that, for $h \in \operatorname{U}(\mathcal H)_{\text{str}}$, we have
\begin{equation*}
{}_hf_{S, \xi}(g) \, = \, 
\operatorname{Re}\left( \langle g^{-1} h S h^{-1} g \xi \mid \xi \rangle \right)
\end{equation*}
for all $g \in \operatorname{U}(\mathcal H)_{\text{str}}$, i.e. 
${}_hf_{S, \xi} = f_{hSh^{-1}, \xi}$.
\par

Consider the linear form
\begin{equation*}
\tau_{\xi} \, : \, \left\{
\aligned
\operatorname{L}(\mathcal H) \hskip.2cm &\longrightarrow \hskip.5cm \R
\\
S \hskip.5cm &\longmapsto \hskip.2cm 
M(f_{S,\xi}) .
\endaligned
\right.
\end{equation*}
Since $M$ is left-invariant, we have,
for all  $S \in \operatorname{L}(\mathcal H)$ 
and $h \in \operatorname{U}(\mathcal H)_{\text{str}}$,
\begin{equation*}
\tau_\xi (hSh^{-1}) \, = \, M(f_{hSh^{-1},\xi}) \, = \, M({}_hf_{S, \xi}) \, = \,
M(f_{S,\xi})  \, = \,   \tau_\xi (S) ,
\end{equation*}
and therefore also $\tau_\xi (Sh) = \tau_\xi (hS)$.
\par

Every operator in $\operatorname{L}(\mathcal H)_{\text{str}}$
is a linear combination of unitaries\footnote{Let $A$ be a C$^*$-algebra with unit.
Every $x \in A$ is a linear combination of four unitaries. 
Indeed, since $x = \frac{1}{2} (x+x^*) + \frac{1}{2i} (ix - ix^*)$,
it is enough to check that a self-adjoint element of norm at most $1$
is a linear combination of two unitaries.
If $x^* = x$ and $\Vert x \Vert \le 1$, 
then $u = x + i\sqrt{1-x^2}$ and $u^* = x - i\sqrt{1-x^2}$
are unitary, and $x = \frac{1}{2}(u + u^*)$.
}.
Hence $\tau_\xi (ST) = \tau_\xi (TS)$ for all
$S,T \in \operatorname{L}(\mathcal H)$.
Since the identity operator is a sum of two commutators,
i.e. since $\operatorname{id}_{\mathcal H}$ 
is of the form $S_1T_1 - T_1S_1 + S_2T_2 - T_2S_2$ 
(see e.g.\ Problem 186 in \cite{Halm--67}),
we have $\tau_\xi (\operatorname{id}_{\mathcal H}) = 0$.
But this is preposterous, because
$\tau_\xi (\operatorname{id}_{\mathcal H}) = M(1) = 1$.
Hence $\operatorname{U}(\mathcal H)_{\text{str}}$ is not B-amenable.
\end{proof}

\begin{rem}
% 5.3
\label{remonU(H)}
(1)
Let  the notation be as in the proof above.
Since the group $\operatorname{U}(\infty)$
is B-amenable (by Claim \ref{2DEhersam} of Proposition \ref{hersam}) 
and dense in $\operatorname{U}(\mathcal H)_{\text{str}}$,
Proposition \ref{Ustrong} justifies Remark \ref{remonheram}(3).

\vskip.2cm

(2) 
Proposition \ref{Ustrong} has the following offspring.
Let $\mathcal M$ be a von Neumann algebra,
realized as a weakly closed $*$-subalgebra of $\operatorname{L}(\mathcal H)$,
for some separable Hilbert space $\mathcal H$.
Let $\operatorname{U} (\mathcal M)_{\text{str}}$ be its unitary group,
with the strong topology.
Then $\operatorname{U} (\mathcal M)_{\text{str}}$ is amenable if and only if
$\mathcal M$ is injective \cite{Harp--79}.
For a C$^*$-algebra $A$, there is a similar characterization of
nuclearity of $A$ in terms of amenability of the unitary group $\operatorname{U} (A)$ of $A$,
with the norm topology
\cite{Pate--92}.

\vskip.2cm

(3)
Consider the Banach algebra $\operatorname{L}(\mathcal H)$
of all bounded operators on $\mathcal H$, with the usual operator norm.
Let $\operatorname{GL}(\mathcal H)_{\text{norm}}$ be its general linear group,
with its topology as an open subset of $\operatorname{L}(\mathcal H)$.
We denote by $\operatorname{U}(\mathcal H)_{\text{norm}}$ the unitary group of $\mathcal H$,
with the topology induced by the operator norm; 
it is a closed subgroup of  $\operatorname{GL}(\mathcal H)_{\text{norm}}$.
The groups $\operatorname{U}(\mathcal H)_{\text{norm}}$ 
and $\operatorname{GL}(\mathcal H)_{\text{norm}}$
are not amenable.
Also, the group $\operatorname{GL}(\ell^p)_{\text{norm}}$ is not amenable,
for $p$ with $1 \le p < \infty$.
For this (and more), see \cite{Harp--73} and \cite[Examples 3.6.13-15]{Pest--06}.
\end{rem}

From the proof of Proposition \ref{Ustrong},
let us extract the following observation.
It appeared in \cite[Page 489]{Harp--82}.

\begin{cor}
% 5.4
\label{closedsubg}
A closed subgroup of an amenable topological group need not be amenable.
\end{cor}

Let $\Sy (\N)$ denote the full \emph{symmetric group} of the positive integers,
with its standard Polish topology
(any infinite countable set would do instead of $\N$).
There is a curiosity similar to that noted just before Lemma \ref{lambda}:
the group $\Sy (\N)$ has a unique topology making it a Polish group
\cite{Gaug--67, Kall--79}, indeed a unique topology making it a separable Hausdorff group
\cite[Theorem 1.11]{KeRo--07}.
We denote by $\Sy_f(\N)$ the subgroup of $\Sy(\N)$ of permutations with finite support;
it is a locally finite group, dense in $\Sy(\N)$.
\par

As in Proposition \ref{Ustrong}, we have: 

\begin{prop}
% 5.5
\label{Sym}
The topological group  $\Sy (\N)$
is amenable and is not B-amenable.
\end{prop}

\begin{proof}
The proof is analogous to that of Proposition \ref{Ustrong}.
On the one hand, since $\Sy(\N)$ contains a dense subgroup 
that is locally finite and therefore amenable, $\Sy(\N)$
is amenable by Proposition \ref{heram}.
On the other hand, $\N$ can be identified (as a set)
with a non-amenable countable group $\Gamma$,
and the metrizable group $\Sy (\Gamma)$ contains $\Gamma$ as a discrete closed subgroup,
so that $\Sy (\N) \simeq \Sy (\Gamma)$ is not B-amenable,
by Corollary \ref{coverings}.
\end{proof}

Let $V$ be a vector space over a finite field $\GF(q)$.
Assume that the dimension of $V$ is infinite and countable.
Observe that $V$ is a countable infinite set, 
and that the \emph{general linear group} $\GL(V)$
is a subgroup of $\Sy(V) \simeq \Sy(\N)$.
\par
Moreover, each of the equations
\begin{enumerate}[noitemsep]
\item[]
$g(0) = 0$, \hskip.2cm with $0$ the origin of $V$,
\item[]
$g(\lambda v) = \lambda g(v)$, \hskip.2cm with $\lambda \in GF(q)^\times$ and $v \in V$,
\item[]
$g(v+w) = g(v) + g(w)$, \hskip.2cm with $v,w \in V$,
\end{enumerate}
defines a closed subset of $\Sy(V)$.
Hence $\GL(V)$ is a closed subgroup of $\Sy(V)$;
in particular $\GL(V)$ itself is a Polish group.

\begin{prop}
% 5.6
\label{GL(V)}
The topological group  $\GL(V)$
is amenable and is not B-amenable.
\end{prop}

\begin{proof}
This is one more variation on the same proof,
as for Propositions \ref{Ustrong} and \ref{Sym}.
Let $(e_n)_{n \in \N}$ be a basis of $V$.
For every $n \in \N$, denote by $V_n$ the linear span of $\{e_1, e_2, \hdots, e_n\}$,
and by $\GL_n$ the subgroup of $\GL (V)$ consisting of those elements $g$
such that $g(V_n) = V_n$ and $g(e_k) = e_k$ for all $k > n$.
We have a nested sequence
\begin{equation*}
\GL_1 \subset \cdots \subset \GL_n \subset \GL_n \subset \cdots \subset
\GL_f :=  \bigcup_{n=1}^\infty \GL_n
\end{equation*}
of which the union $\GL_f$ is locally finite, in particular amenable,
and dense in $\GL(V)$. Hence $\GL(V)$ is amenable.
\par
The space $V$ has a basis $(e_\gamma)_{\gamma \in F}$
indexed by a non-abelian free group $F$.
Each $\gamma_0 \in F$ can be viewed as an element of $\GL(V)$
mapping $e_{\gamma}$ to $e_{\gamma_0 \gamma}$ for all $\gamma \in F$.
This shows that $\GL(V)$ contains a discrete closed subgroup isomorphic to $F$,
and in particular that $\GL (V)$ is not B-amenable.
\end{proof}

\begin{exe}
% 5.7
Let $\mathbf k$ be a commutative ring, with unit.
Denote by $\mathcal J (\mathbf k)$ 
the {\it substitution group of formal power series over $\mathbf k$},
with elements of the form $f(x) = x + \sum_{i \ge 2} a_i x^i$, where $a_i \in \mathbf k$,
and with substitution for the group law.
For what follows, and much more, about groups of this kind, see \cite{Babe--13}.
\par
For each $n \ge 1$, denote by $\mathcal J^{n+1}(\mathbf k)$
the normal subgroup of $\mathcal J (\mathbf k)$ defined by the equations
$a_2 = a_3 = \cdots = a_{n+1} = 0$ and by $\mathcal J_n(\mathbf k)$
the quotient $\mathcal J(\mathbf k) / \mathcal J^{n+1}(\mathbf k)$.
There is a natural bijection between $\mathcal J_n(\mathbf k)$ and $\mathbf k^n$.
When $\mathbf k$ is a topological ring, we use this bijection to define 
a topology on $\mathcal J_n(\mathbf k)$; it is a group topology.
Then $\mathcal J(\mathbf k)$ is also a topological group,
with the topology of the inverse limit $\varprojlim_n \mathcal J_n(\mathbf k)$.
This topology is interesting even if the ring $\mathbf k$ is discrete;
note that the group $\mathcal J(\mathbf k)$ is profinite 
when the ring $\mathbf k$ is finite.
\par
It is known that, for every topological commutative ring $\mathbf k$,
the group $\mathcal J(\mathbf k)$ is amenable \cite{BaBo--11}.
According to \cite[Page 61]{Babe--13},
it is not known whether the group $\mathcal J(\mathbf Z)$
is B-amenable.
\end{exe}

\section{\textbf{Extreme amenability}}
% s6
\label{extreme}

Several of the examples we know of topological groups
that are amenable and not locally compact
have a property stronger than amena\-bility: that of extreme amenability.
The notion appeared in the mid 60~'s.
The subject became more important with later articles,
such as those of Gromov-Milman \cite{GrMi--83}, written in the late 70 's,
and Kechris-Pestov-Todorcevic \cite{KePT--05}.
For indications on the development of the subject,
see the introduction of \cite{Pest--06}.

\begin{defn}
% 6.1
\label{DefExAm}
A topological group $G$ is \textbf{extremely amenable} if
every continuous action of $G$
on a compact space has a fixed point.
\end{defn}

In the next proposition, Items \ref{1DEstabea} to \ref{4DEstabea} 
are reminiscent of similar items
in Propositions \ref{heram} and \ref{hersam}.

\begin{prop}[hereditary properties of extreme amenability]
% 6.2
\label{stabea}
Let $G$ be a topological group.
\begin{enumerate}[noitemsep,label=(\arabic*)]
\addtocounter{enumi}{-1}
\item\label{0DEstabea}
If $G$ is extremely amenable, then $G$ is ame\-nable.
\item\label{1DEstabea}
If $G$ is extremely amenable, every open subgroup of $G$
is extremely amenable. 
\item\label{2DEstabea}
If $G$ is a directed union of a family $(H_\alpha)_{\alpha \in A}$ of closed subgroups,
and if each $H_\alpha$ is extremely amenable, then $G$ is extremely amenable.
\item\label{3DEstabea}
If $G$ has an extremely amenable closed normal subgroup $N$ such that
the quotient $G/N$ is extremely amenable, then $G$ is extremely amenable.
\item\label{4DEstabea}
Let $H$ be a topological group such that there exists
a continuous homomorphism $H \longrightarrow G$ with dense image;
if $H$ is extremely amenable, then so is $G$.
\item\label{5DEstabea}
A closed subgroup of an extremely amenable group NEED NOT be  amenable.
\end{enumerate}
\end{prop}

\begin{proof}
Claims \ref{0DEstabea} and \ref{4DEstabea} are straightforward. 
Claim \ref{1DEstabea} appears in \cite[Lemma 13]{BoPT--11}.
\par

\ref{2DEstabea}
Let $X$ be a non-empty compact $G$-space. For each $\alpha \in A$,
the set $X^\alpha$ of $H_\alpha$-fixed points is closed and therefore compact in $X$,
and non-empty by hypothesis on $H_\alpha$.
The intersection $\bigcap_{\alpha \in A} X^\alpha$ is non-empty,
by compactness of $X$.
Since this intersection coincides with the set of $G$-fixed points in $X$,
the group $G$ is extremely amenable.
\par

For \ref{3DEstabea}, see \cite[Corollary 6.2.10]{Pest--06}.
\par

A smart example confirming \ref{5DEstabea} is the group
$\operatorname{Aut}(\Q)$ of order-preser\-ving permutations 
of the set $\Q$ of rational numbers
(see next proposition), 
which contains a non-abelian free discrete subgroup
\cite[Theorem 8.1]{Pest--98}.
Note that a non-trivial locally compact closed subgroup of an extremely amenable group
\emph{is not} extremely amenable, see Proposition \ref{NExExAm} below.
\end{proof}

\begin{prop}
% 6.3
\label{ExExAm}
The following groups are extremely amenable.
\begin{enumerate}[noitemsep,label=(\arabic*)]
\item\label{1DEExExAm}
The unitary group $\operatorname{U} (\mathcal H)_{\text{str}}$
of an infinite dimensional separable complex Hilbert space $\mathcal H$,
with the strong topology.
\item\label{2DEExExAm}
The group $\operatorname{Aut}(\Q)$
of order-preserving permutations of  $\Q$,
with the topology of pointwise convergence on the discrete space $\Q$.
\item\label{3DEExExAm}
The group $\mathcal H^+ ([0,1])$ 
of orientation-preserving homeomorphisms of the closed unit interval,
with the compact-open topology.
\item\label{4DEExExAm}
The isometry group of the Urysohn space.
\item\label{5DEExExAm}
The group $\operatorname{Aut} (X,\mu)$ of all measure-preserving automorphisms
of a standard non-atomic finite or infinite and sigma-finite measure space,
with the weak topology.
\item\label{6DEExExAm}
The group $L^0 (X, \mathcal B, \mu; G)$ of all measurable maps 
from a Lebesgue space with a non-atomic probability measure $(X, \mathcal B, \mu)$
to a second-coun\-table compact group $G$,
up to equality $\mu$-almost everywhere,
with the topology of convergence in measure.
\end{enumerate}
\end{prop}

\begin{proof}[References for the proof]
The case of $\operatorname{U} (\mathcal H)_{\text{str}}$ 
is shown in \cite{GrMi--83};
for the translation of their result in terms of Definition \ref{DefExAm}, 
see for example \cite[Section 2.2]{Pest--06}.
For $\operatorname{Aut}(\Q)$ and $\mathcal H^+ ([0,1])$, see \cite{Pest--98}.
\par

For the Urysohn space and its isometry group, see \cite{Pest--06},
in particular Theorem 5.3.10 (the isometry group of the Urysohn space,
with its standard Polish topology, is a Levy group)
and Theorem 4.1.3 (every Levy group is extremely amenable).
\par

For $\operatorname{Aut} (X,\mu)$, see \cite[Theorem 4.2]{GiPe--07}.
The case of $L^0 (X, \mathcal B, \mu; G)$ is due to Eli Glasner \cite{Glas--98}
and Furstenberg-Weiss [unpublished]; 
see \cite[Section 4.2]{Pest--06}.
\par

For some of these groups, extreme amenabiliy can also be shown
by the arguments of \cite{MeTs--13}.
\end{proof}

\begin{rem}
% 6.4
\label{(T)..Thompson}
Propositions \ref{Ustrong} and \ref{ExExAm}\ref{1DEExExAm}
show that an extremely amenable group need not be B-amenable.
\par

Concerning \ref{1DEExExAm},
recall on the one hand that the group $\operatorname{U}(\mathcal H)_{\text{str}}$
is known to have Kazhdan's Property (T) \cite{Bekk--03}.
On the other hand,
an amenable \emph{locally compact} group which has Property (T) is compact.

\vskip.2cm

Concerning \ref{2DEExExAm},
note that not only $\operatorname{Aut}(\Q)$ with the indicated topology
is extremely amenable,
but moreover every action by homeomorphisms
of the group $\operatorname{Aut}(\Q)$ with the \emph{discrete topology}
on a compact \emph{metrizable space} has a fixed point \cite[Corollary 7]{RoSo--07}.

\vskip.2cm

Concerning \ref{3DEExExAm}, recall that
Thompson's group $F$ is a dense subgroup of $\mathcal H^+ ([0,1])$.
The only consequence of \ref{3DEExExAm} we can state is that is does not exclude
that $F$ is amenable. 
(This is a repetition of Remark 12 of \cite{Pest--02}.)
\end{rem}

For the next proposition, we use the following notation.
\par
$\Sy (\N)$ is the symmetric group of  $\N$,
with its standard Polish topology, as in Proposition \ref{Sym}.
\par

Let $p$ be a positive number with $1 \le p < \infty$ and $p \ne 2$.
Let $\ell^p$ be the Banach space of sequence $(z_n)_{n \ge 1}$ of complex numbers
such that $\Vert z \Vert := \left(\sum_{n \ge 1} \vert z_n \vert ^p \right)^{1/p} < \infty$.
We denote by $\operatorname{U}(\ell^p)$ the group of linear isometries of $\ell^p$, 
with the strong topology.

\begin{prop}
% 6.5
\label{NExExAm}
The following groups are NOT extremely amenable.
\begin{enumerate}[noitemsep,label=(\arabic*)]
\item\label{1DENExExAm}
Any locally compact group $G \ne \{1\}$.
\item\label{2DENExExAm}
The symmetric group $\Sy (\N)$.
\item\label{3DENExExAm}
The group $\mathcal H (C)$ 
of homeomorphisms of the Cantor space,
with the compact-open topology.
\item\label{5DENExExAm}
The unitary group $\operatorname{U}(\ell^p)$,
for $p$ with $1 \le p < \infty$ and $p \ne 2$.
\item\label{6DENExExAm}
The group $\operatorname{GL}(V)$, for $V \simeq \GF(q)^{(\N)}$ 
as in Proposition \ref{GL(V)}.
\end{enumerate}
\end{prop}

\begin{proof}[References for the proof]
Every locally compact group admits a free action on a suitable compact space \cite{Veec--77};
and \ref{1DENExExAm} follows.

For \ref{2DENExExAm} and $\Sy (\N)$, 
see \cite{Pest--98}, or \cite[Section 2.4]{Pest--06}.
As Glasner and Weiss have observed, the natural action of $\Sy (\N)$ 
on the compact space $\operatorname{LO} \subset \N^2$ of all linear orders on $\N$
has no fixed point; 
details in \cite[Example 2.4.6, Page 47]{Pest--06}.
\par

For \ref{3DENExExAm}, observe that the action of $\mathcal H (C)$ on $C$
has no fixed point.
For more on actions of this group on compact spaces, see \cite{GlWe--03}. 
\par

For \ref{5DENExExAm}, let us reproduce
the argument of \cite[Example 3.6.15]{Pest--06}.
Let $(e_n)_{n \in \N}$ be the canonical basis of $\ell^p$.
For every sequence $(t_n)_{n \in \N}$ in the compact group $\prod_{n \in \N} T_n$,
where each $T_n$ is a copy of $T = \left\{ z \in \C \mid \vert z \vert = 1 \right\}$, 
and every $\sigma \in \Sy (\N)$, there is an isometry of $\ell^p$
mapping $e_n$ to $t_n e_{\sigma(n)}$ for all $n \in \N$.
Moreover, every isometry of $\ell^p$ is of this form; 
the proof is like that of Banach, for the $\ell^p$-space of real sequence
and for $t_n \in \{-1,1\}$ for all $n$ \cite[Chap.\ XI, $\S$~5, Page 178]{Bana--32}.
In other terms, the unitary group $\operatorname{U}(\ell^p)$ of $\ell^p$
is a semi-direct product
\begin{equation*}
\operatorname{U}(\ell^p) \, = \, \Big( \prod_{n \in \N} T_n \Big) \rtimes \Sy (\N) .
\end{equation*}
Now Claim \ref{5DENExExAm} 
follows from \ref{2DENExExAm} 
and Proposition \ref{stabea}\ref{3DEstabea}.
\par

For \ref{6DENExExAm}, see 
\cite[Example 6.7.17]{Pest--06}.
\end{proof}

To conclude this section, we quote a result
that extends Proposition \ref{ExExAm}\ref{1DEExExAm}.
Compare with the way Remark \ref{remonU(H)}(2)
extends part of Proposition \ref{Ustrong}.

\begin{prop}
% 6.6
\label{UvonnA}
A countably decomposable
von Neumann algebra $M$ is injective 
if and only if its unitary group $\operatorname{U}(M)$,
with the weak topology,
is the direct product of a compact group and an extremely amenable group.
\par
In particular, an infinite dimensional factor $M$ is injective
if and only if  $\operatorname{U}(M)$
is extremely amenable,
and the same holds for $M$ a properly infinite von Neumann algebra.
\end{prop}

\begin{proof}[References for the proof:]
\cite[Theorem 3.3]{GiPe--07} and \cite{GiNg--13}.
\end{proof}

\section{\textbf{On the definition of ergodicity}}
% s7
\label{defergetc}

Let us agree on the following terminology.
Consider a Borel action of a topological group $G$ 
on a Borel space $(X, \mathcal B)$.
A Borel subset $A \subset X$ is \textbf{invariant} by $G$
if $gA = A$ for all $g \in G$.
\par

Assume that we have moreover 
a $G$-invariant probability measure $\mu$ on $(X, \mathcal B)$.
A Borel subset $A \subset X$ is 
\textbf{$\mu$-essentially invariant} by $G$ if 
$\mu( gA \hskip.1cm \Delta \hskip.1cm A) = 0$ for all $g \in G$
(where $\Delta$ indicates a symmetric difference).
For ``$\mu$-essentially invariant'', 
Varadarajan uses ``$\mu$-invariant'' \cite[Page 196]{Vara--63},
Maitra ``$\mu$-almost invariant''  \cite{Mait--77}, 
and Phelps ``invariant (mod $\mu$)'' \cite{Phel--66}.

\begin{defn}
% 7.1
Let $G$ be a topological group acting as above on a Borel space $(X, \mathcal B)$
with a $G$-invariant probability measure $\mu$.
\par

The action is \textbf{w-ergodic} if,
for every invariant set $A \in \mathcal B$, 
either $\mu(A) = 0$ or $\mu(X \smallsetminus A) = 0$;
in this situation, we say also that the invariant measure $\mu$ is w-ergodic.
\par

The action is \textbf{s-ergodic} if, 
for every $\mu$-essentially invariant set $A \in \mathcal B$, 
either $\mu(A) = 0$ or $\mu(X \smallsetminus A) = 0$;
in this situation, we say also that the invariant measure $\mu$ is s-ergodic.
\end{defn}

Some authors use ``ergodic'' for our ``w-ergodic''
(see for example \cite[beginning of Chapter 2]{Zimm--84}),
others use ``ergodic'' for our ``s-ergodic''
(see for example \cite[Chapter 10]{Phel--66}).
Proposition \ref{werg=serg} shows that, in a standard setting, the two notions coincide.

\begin{rem}
% 7.2
\label{banalite}
Consider a Borel action of a topological group $G$ on a Borel space $(X, \mathcal B)$,
and a $G$-invariant probability measure $\mu$ on $X$.
Assume that $\mu$ is s-ergodic. 
Then every Borel subset $A$ in $X$ such that $\mu(A) = 1$ is 
$\mu$-essentially invariant by $G$.
\par
Indeed, let $g \in G$.
The subset $g(A) \smallsetminus (A \cap g(A))$ is contained in $X \smallsetminus A$,
hence is negligible for $\mu$.
Similarly, 
$A \smallsetminus (A \cap g(A)) = g \big( g^{-1}(A) \smallsetminus (A \cap g^{-1}(A)) \big)$
is negligible for $\mu$, because $\mu$ is $G$-invariant.
Hence $\mu ( A \Delta g(A)) = 0$ for all $g \in G$.
\end{rem}

Suppose moreover that $X$ is a compact space, 
and that $\mathcal B$ is the $\sigma$-algebra of Borel subsets of $X$,
i.e.\ the $\sigma$-algebra of subsets of $X$ generated by the open subsets of $X$.
The space $\mathcal P (X)$ of probability measures on $X$ is a compact convex subspace
of the dual space of $\mathcal C (X)$, with the weak-$*$-topology,
and  the space $\mathcal P^G (X)$ of $G$-invariant probability measures
is a compact convex subspace of $\mathcal P (X)$.

\begin{defn}
% 7.3
\label{defindecmeas}
With the notation above, a $G$-invariant measure $\mu \in \mathcal P^G (X)$ 
is \textbf{indecomposable} if it is in the subset  of extreme points $\mathcal E^G (X)$
of $\mathcal P^G (X)$.
\par
In other words, $\mu$ is \textbf{decomposable} if there exist
two distinct measures $\mu_1,\mu_2 \in  \mathcal P^G (X)$
and a constant $c$ with $0 < c < 1$ such that $\mu = c \mu_1 + (1-c)\mu_2$.
\end{defn}

The following proposition appears in \cite{Bogo--39}.

\begin{prop}
% 7.4
\label{FominEq}
Let $G$ be a topological group acting continuously by homeomorphisms
on a compact space $X$, and let $\mu$ be a $G$-invariant probability measure on $X$.
The following two properties are equivalent:
\begin{enumerate}[noitemsep,label=(\roman*)]
\item\label{iDEFominEq}
$\mu$ is indecomposable,
\item\label{iiDEFominEq}
$\mu$ is s-ergodic;
\end{enumerate}
and imply the following third property
\begin{enumerate}[noitemsep,label=(\roman*)]
\addtocounter{enumi}{2}
\item\label{iiiDEFominEq}
$\mu$ is w-ergodic.
\end{enumerate}
\end{prop}

\emph{Note.} 
We will add another property equivalent to \ref{iDEFominEq} and \ref{iiDEFominEq}
in Proposition \ref{FominTh2}.
\par

It is recalled below that \ref{iiiDEFominEq} does imply \ref{iiDEFominEq} when 
$G$ is second countable locally compact (Proposition \ref{werg=serg}), 
but not in general (Proposition \ref{PropBruno}\ref{3DEPropBruno}).

\begin{proof} 
For \ref{iDEFominEq} $\Leftrightarrow$ \ref{iiDEFominEq}, 
we reproduce the proof of \cite{Bogo--39}.
This proof has remained the standard one: see for example \cite[Theorem 610]{Walt--82}.
Walters writes his proof for one continuous map $T : X \longrightarrow X$,
but it works without change in the present situation. 
For \ref{iiDEFominEq} $\Rightarrow$ \ref{iDEFominEq}, 
see also \cite[Proposition 10.4]{Phel--66}
(Proposition 12.4 of the Second Edition).

\vskip.2cm

\ref{iDEFominEq} $\Rightarrow$ \ref{iiDEFominEq}
We prove the contraposition.
Assume that there exists
a $\mu$-essentially $G$-invariant subset $A$ of $X$ with $0 < \mu(A) < 1$.
Then $\mu = \mu(A) \mu_1 + (1-\mu(A)) \mu_2$, with
\begin{equation*}
\mu_1(B) \, = \, \frac{1}{\mu(A)} \mu( B \cap A), \hskip.5cm
\mu_2(B) \, = \, \frac{1}{1-\mu(A)} \mu( B \cap (X \smallsetminus A)),
\end{equation*}
for all Borel subsets $B$ of $X$.

\vskip.2cm

\ref{iiDEFominEq} $\Rightarrow$ \ref{iDEFominEq}
Assume (again by contraposition) that there exist
two distinct $G$-invariant probability measures $\mu_1, \mu_2$
of which $\mu = c_1 \mu_1 + c_2 \mu_2$ is a convex combination.
For every Borel subset $B$ of $X$ with $\mu(B) = 0$, we have $\mu_1(B) = 0 = \mu_2(B)$,
i.e.\ $\mu_1, \mu_2$ are absolutely continuous with respect to $\mu$.
By the Radon-Nikodym theorem 
(one of many convenient references is \cite[Theorem 6.9]{Rudi--66}), 
there exist well-defined and unique positive-valued functions $f_1, f_2 \in L^1(X,\mu)$
such that $\mu_1 = f_1\mu$ and $\mu_2 = f_2\mu$.
By uniqueness, $f_1$ and $f_2$ are $G$-invariant.
Set
\begin{equation*}
A \, = \, \{ x \in X \mid f_1(x) > f_2(x) \}
\hskip.2cm \text{and} \hskip.2cm
B \, = \, \{ x \in X \mid f_1(x) \le f_2(x) \} .
\end{equation*}
Then $A,B$ are $\mu$-essentially $G$-invariant Borel
subsets of $X$ of positive measure, and constitute a partition of $X$.

\vskip.2cm

The implication \ref{iiDEFominEq} $\Rightarrow$ \ref{iiiDEFominEq} is trivial.
\end{proof}

\begin{defn}
% 7.5
\label{defessGinv}
Let $(\Omega, \mathcal B)$ be a Borel space;
assume that it is a standard Borel space, i.e.\ that there exist an isomorphism
of $(\Omega, \mathcal B)$ with a Borel subset of a complete separable metric space.
Let $G$ be a topological group acting in a Borel way on $(\Omega,\mathcal B)$;
assume that there exists a probability measure $\mu$ on $(\Omega,\mathcal B)$
which is a \textbf{$G$-quasi-invariant}, i.e. such that,
for all $A \subset \mathcal B$ and $g \in G$,
we have $\mu(A) = 0$ if and only if $\mu(gA) = 0$.
\par

Let  $(Y, \mathcal C)$ be another standard Borel space
(the most important case here is $Y = \mathbf R$, with the usual Borel $\sigma$-algebra).
A Borel function $f : \Omega \longrightarrow Y$ is \textbf{$\mu$-essentially $G$-invariant} if,
for each $g \in G$, we have $f(g \omega) = f(\omega)$ 
for $\mu$-almost all $\omega \in \Omega$;
it is \textbf{$G$-invariant} if, for each $g \in G$, we have $f(g \omega) = f(\omega)$ 
for all $\omega \in \Omega$.
\end{defn}

\begin{lem}[Lemma 3.3 in \cite{Vara--63}, or Lemma 2.2.16 in \cite{Zimm--84}]
% 7.6
\label{VaraZim}
Let $G$ be a second countable locally compact group
acting in a Borel way on a standard Borel space $\Omega$.
Assume that $\Omega$ has a $G$-quasi-invariant probability measure $\mu$.
\par

Let $Y$ be a standard Borel space
and $f : \Omega \longrightarrow Y$ be a Borel function.
Assume that $f$ is essentially $G$-invariant.
Then there exists a Borel function $\tilde f : \Omega \longrightarrow Y$
which is $G$-invariant, and $\tilde f(\omega) = f(\omega)$ 
for $\mu$-almost all $\omega \in \Omega$.
\par

In particular, if $A \subset \Omega$ is a $\mu$-essentially $G$-invariant Borel subset,
there exists a $G$-invariant Borel set $\widetilde A \subset \Omega$
such that $\mu (\widetilde A \Delta A) = 0$.
\end{lem}

\begin{proof}[On the proof]
Note that the particular case follows from the result on functions,
with  $Y = \mathbf R$ and $f$ the characteristic function of $A$.
\par
It is crucial here that $G$ is locally compact (hence $G$ has a Haar measure)
and ``not too large'' (more precisely ``second countable''
in \cite{Zimm--84} and  \cite{Vara--63}),
because the proof relies very strongly on Fubini's theorem,
and Fubini's theorem holds for Haar measure with appropriate conditions only.
\end{proof}

\begin{prop}
% 7.7
\label{werg=serg}
Let $G$ be a second countable locally compact group
acting in a Borel way on a standard Borel space $\Omega$.
Assume that $\Omega$ has a $G$-quasi-invariant probability measure $\mu$.
\par

Then the action is s-ergodic if and only if it is w-ergodic.
\end{prop}

\begin{proof}
The proposition follows from the lemma and the definitions.
\end{proof}

We quote now the following decomposition theorem.
It can be seen as an elaboration of results going back to von Neumann, in the early 30's.
Two convenient references are \cite{Vara--63}, for invariant measures,
and \cite{GrSc--00}, for quasi-invariant measures.

\begin{thm}[Ergodic Decomposition Theorem]
% 7.8
\label{EDT}
Let $G$ be a second countable locally compact group 
acting in a Borel way on a standard Borel space $(X, \mathcal B)$;
assume that $\mathcal P (X)^G$ is non-empty 
(equivalently that $\mathcal E (X)^G$ is non-empty).
Denote by $(Y, \mathcal C)$ the standard Borel space
with $Y = \mathcal E(X)^G$ and $\mathcal C$ its Borel $\sigma$-algebra.
\par
Then there exist a family $(p_y)_{y \in Y}$ of probability measures on $X$,
with the following properties:
\begin{enumerate}[noitemsep,label=(\arabic*)]
\item\label{1DEEDT}
for every Borel subset $B$ in $X$, 
the map $y \longmapsto p_y(B)$ from $Y$ to $[0,1]$ is Borel;
\item\label{2DEEDT}
for every $y \in Y$, the measure $p_y$ is $G$-invariant
and ergodic;
\item\label{3DEEDT}
for $y,y' \in Y$ with $y \ne y'$, the measures $p_y$ and $p_{y'}$
are mutually singular;
\item\label{4DEEDT}
for every $\mu \in \mathcal P (X)^G$, 
there exists a probability measure $\nu$ on $(Y, \mathcal C)$ such that
$\mu(B) = \int_Y p_y(B) d\nu (y)$ for every $B \in \mathcal B$.
\end{enumerate}
\end{thm}

\emph{Note:} a fortiori, the theorem holds for a continuous action of $G$
on a metrizable compact space.

\begin{cor}
% 7.9
\label{cor1EDT}
Let $G$ be a second-countable locally compact group
acting continuously by homeomorphisms on a metrizable compact space $X$.
\par

For every $\mu \in \mathcal E^G(X)$, there exists a Borel subset $A_\mu$ of $X$
such that
\begin{equation*}
\mu(A_\mu) \, = \, 1 \hskip.5cm \text{and} \hskip.5cm
\mu'(A_\mu) \, = \, 0 \hskip.2cm \text{for all} \hskip.2cm  \mu' \in \mathcal E^G(X)
\hskip.2cm \text{with} \hskip.2cm \mu' \ne \mu .
\end{equation*}
\end{cor}

\begin{proof}[On the proof]
This follows from a version of the Ergodic Decomposition Theorem 
more comprehensive than that of Theorem \ref{EDT}, see \cite{GrSc--00}.
\par
Theorem 1.4 of \cite{GrSc--00} is such a theorem 
for a \emph{countable} discrete group, say $\Gamma$.
It is shown there that 
\begin{enumerate}[noitemsep,label=(\alph*)]
\item\label{aDEcor1EDT}
the sub-$\sigma$-algebra $\mathcal T$ of $\mathcal B$
consisting of $\Gamma$-invariant sets is countably generated,
say generated by a countable sub-algebra $\mathcal T'$;
\item\label{bDEcor1EDT}
for every $x \in X$, the set $[x]_{\mathcal T} = \bigcap C$, 
where the intersection is over all $C \in \mathcal T'$
with $x \in C$, is a Borel subset of $X$
(and it coincides with the intersection $\bigcap C$
over all $C \in \mathcal T$ with $x \in C$);
\item\label{cDEcor1EDT}
$[x]_{\mathcal T} \in \mathcal T$, i.e.\ 
$[g(x)]_{\mathcal T} = [x]_{\mathcal T}$ for all $x \in X$ and $g \in \Gamma$;
\item\label{dDEcor1EDT}
there exist a $\Gamma$-invariant Borel subset $X_0 \subset X$,
with $\eta(X \smallsetminus X_0) = 0$ for all $\eta \in \mathcal P^\Gamma (X)$,
and a surjective map $p : X_0 \longrightarrow Y_\Gamma$, $x \longmapsto p_x$, 
measurable with respect to $\mathcal T$ and $\mathcal C$, such that
\begin{enumerate}[noitemsep]
\item[]
$p_x\left( [x]_{\mathcal T} \right) = 1$,
\item[]
$p_{x'}\left( [x]_{\mathcal T} \right) = 0$,
\end{enumerate}
for all $x,x' \in X_0$ with $p(x') \ne p(x)$.
\end{enumerate}
Then, for $\mu \in \mathcal E^\Gamma (X)$, it suffices to set:
\begin{enumerate}[noitemsep,label=(\alph*)]
\addtocounter{enumi}{4}
\item\label{eDEcor1EDT}
$A_\mu = [x]_{\mathcal T}$, with $x \in X_0$ such that $p(x) = \mu$.
\end{enumerate}
[Note that, in \cite{GrSc--00}, the set $[x]_{\mathcal T}$ is defined
as an intersection over $C \in \mathcal T$ with $x \in C$;
but $[x]_{\mathcal T}$ is Borel because this is also 
an intersection over $C$ in the countable subalgebra $\mathcal T'$.
We are grateful to K.~Schmidt for an e-mail clarifying this point for us.]
\par

In the more general case of a second-countable locally compact group $G$,
we refer to \cite[Theorem 5.2]{GrSc--00}.
Choose a countable dense subgroup $\Gamma$ of $G$.
Define $X_0$ as above, in terms of $\Gamma$ and,
for $x \in X_0$, define also $p_x$ and $[x]_{\mathcal T}$  in terms of $\Gamma$.
Then we have:
\begin{enumerate}[noitemsep,label=(\alph*)]
\addtocounter{enumi}{5}
\item\label{fDEcor1EDT}
$[g(x)]_{\mathcal T} = [x]_{\mathcal T}$ for all $g \in G$ and $x \in X$
(compare with \ref{cDEcor1EDT});
\item\label{gDEcor1EDT}
$\mathcal E^G (X) = \mathcal E^\Gamma (X)$;
\item\label{hDEcor1EDT}
$p_x \in \mathcal Y = \mathcal E^G(X)$ (compare with \ref{dDEcor1EDT});
\item\label{iDEcor1EDT}
$p_x\left( [x]_{\mathcal T} \right) = 1$ and 
$p_{x'}\left( [x]_{\mathcal T} \right) = 0$ for all $x,x' \in X_0$ with $p(x') \ne p(x)$
(as in \ref{dDEcor1EDT}).
\end{enumerate}
[The sets $[x]_{\mathcal T}$ may depend on the choice of $\Gamma$,
but we will not discuss this further here.]

\vskip.2cm

The article \cite{GrSc--00} is written for the case of quasi-invariant measures.
For the particular case of invariant measures, as discussed in this article,
we could equally have quoted an earlier article by Varadarajan.
Specifically, we refer to \cite[Theorem 4.2]{Vara--63} for the fact that
our Borel space $(Y, \mathcal C)$ is standard.
Our sets $A_\mu$, with $\mu \in \mathcal E^G(X)$,
correspond to the sets $X_e$, with $e$ in the space $\mathcal J$
of ergodic $G$-invariant measures, in \cite{Vara--63}.
\end{proof}

We end this section with a simpler version of the previous corollary, 
for comparison with Propositions \ref{KolpasLC} and \ref{6deFomin} below.

\begin{cor}
% 7.10
\label{cor2EDT}
Let $G$ be a second-countable locally compact group
acting continuously by homeomorphisms on a metrizable compact space $X$.
Assume that there exist two distinct ergodic probability measures 
$\mu_1, \mu_2 \in \mathcal E ^G (X)$.
\par
Then there exists two $G$-invariant Borel subsets $A_1, A_2 \subset X$ such that
\begin{equation}
\label{eqcor2EDT}
\aligned
\mu_1(A_1) \, = \, 1 , \hskip.5cm &\mu_1(A_2) \, = \, 0
\\
\mu_2(A_1) \, = \, 0 , \hskip.5cm &\mu_2(A_2) \, = \, 1 .
\endaligned
\end{equation}
\end{cor}

\begin{proof}
This is a consequence of Theorem \ref{EDT},
since, on the one hand, two distinct measures in $\mathcal E^G(X)$ are mutually singular,
and, on the other hand,  two measures $\mu_1$ and $\mu_2$ are mutually singular
precisely when there exist two Borel subsets $A_1, A_2$ in $X$
for which the equalities of (\ref{eqcor2EDT}) hold.
\end{proof}

\section{\textbf{Kolmogorov's example}}
% s8
\label{Kolmo}  

The example of this section is due to Kolmogorov. 
It appears in \cite{Fomi--50};
it has been revisited in \cite{Vara--63}, with reference to Kolmogorov,
and in \cite{Mait--77}, without.
Note that Fomin was a student of Kolmogorov \cite{Ale+--76}.
\par

Consider a positive number $p$ with $0 \le p \le 1$
and the measure $\lambda_p$ on $Y := \{0,1\}$ 
defined by $\lambda_p (0) = p$ and $\lambda_p (1) = 1-p$.
For $i \in \Z$, let $(Y_i, \lambda_{p,i})$ be a copy of $(Y,\lambda_p)$.
Let $X = \prod_{i \in \Z} X_i = \{0,1\}^{\Z}$ be the product of the $Y_i$ 's,
and $\mathcal B$ the usual $\sigma$-algebra;
$(X, \mathcal B)$ is a standard Borel space.
Let $\mu_p$ be the product probability measure 
$\prod_{i \in \Z} \lambda_{p,i}$,
which is called a Bernoulli measure.
We denote by $S : X \longrightarrow X$ the corresponding \textbf{Bernoulli shift},
defined by $(Sx)_i = x_{i+1}$ for all $i \in \Z$;
observe that the measure $\mu_p$ is preserved by $S$.
Recall that the transformation $S$ of $(X, \mu_p)$ is s-ergodic,
indeed strongly mixing; see e.g.\ \cite[Theorems 1.12 and 1.30]{Walt--82}.
\par

Denote as in Proposition \ref{Sym} by $\Sy (\Z)$ the full symmetric group of $\Z$,
with its standard Polish topology,
and by $\Sy_f(\Z)$ the subgroup of $\Sy (\Z)$ of permutations with finite support.
Recall that $\Sy_f (\Z)$ is countable, locally finite, and dense in $\Sy (\Z)$.
The natural action of $\Sy (\Z)$ on $X$ is continuous, and preserves $\mu_p$.
\par

Denote as in Definition \ref{defindecmeas} by $\mathcal P^{\Sy_f(\Z)}(X)$  
the compact convex set of $\Sy_f(\Z)$-invariant probability measures on $X$;
it is a compact convex set in the dual space of $\mathcal C (X)$,
with the weak-$*$-topology,
and $\mathcal E^{\Sy_f(\Z)}(X)$ is the set of its extreme points.
Similarly for $\mathcal E^{\Sy(\Z)}(X)$.
Since $\Sy_f( \Z)$ is a subgroup of $\Sy (\Z)$, 
the space $\mathcal P^{\Sy( \Z)}(X)$ is contained in $\mathcal P^{\Sy_f( \Z)}(X)$.
Since $\Sy_f( \Z)$ is dense in $\Sy (\Z)$, 
and the latter acts continuously on $\mathcal P (X)$, we have indeed
\begin{equation*}
\mathcal P^{\Sy( \Z)}(X) \, = \, \mathcal P^{\Sy_f( \Z)}(X) .
\end{equation*}
In terms of extreme points of compact convex sets, we have consequently
\begin{equation}
% Eq 8.1
\label{eqE=E}
\mathcal E^{\Sy( \Z)}(X) \, = \, \mathcal E^{\Sy_f( \Z)}(X) .
\end{equation}

\begin{prop}
% 8.1
\label{PropBruno}
Consider as above the natural actions of the Polish group $\Sy (\Z)$
and of its subgroup $\Sy_f(\Z)$
on the compact space $X = \{0,1\}^{\Z}$.
\begin{enumerate}[noitemsep,label=(\arabic*)]
\item\label{1DEPropBruno}
The set $\mathcal E^{\Sy_f(\Z)}(X)$ coincides with the set 
$\{ \mu_p \}_{0 \le p \le 1}$ of Bernoulli measures.
\item\label{2DEPropBruno}
Similarly, $\mathcal E^{\Sy(\Z)}(X) = \{ \mu_p \}_{0 \le p \le 1}$;
in particular, for every $p \in [0,1]$, the Bernoulli measure $\mu_p$
is invariant and s-ergodic for the action of $\Sy(\Z)$.
\item\label{3DEPropBruno}
Consider an integer $k \ge 1$, 
a finite sequence $(p_j)_{j=1, \hdots, k}$ with $0 < p_1 < \cdots < p_k < 1$,
positive constants $c_1, \hdots, c_k$ with $c_1 + \cdots + c_k = 1$,
and the $\Sy (\Z)$-invariant probability measure 
$\mu = c_1\mu_{p_1} + \cdots + c_k \mu_{p_k} \in \mathcal P ^{\Sy (\Z)} (X)$.
Then $\mu$ is w-ergodic.
\par
If $k \ge 2$, $\mu$ is not s-ergodic.
\end{enumerate}
\end{prop}

\begin{proof}
Claim \ref{1DEPropBruno} is known as a result of de Finetti.
The original article seems to be \cite{dFin--37}, 
but we rather refer to
\cite[Section VII.4, Pages 228--229]{Fell--71}.
\par

Claim \ref{2DEPropBruno} follows by Equality (\ref{eqE=E}).
It is the way de Finetti's result is quoted in \cite[Section 2.4]{Fomi--50}.
\par

For \ref{3DEPropBruno}, observe that
the $\Sy (\Z)$-orbits on $X$ are easily described:
\begin{enumerate}[noitemsep,label=(\alph*)]
\item\label{1DEBernoulliSym-w-erg}
two one-point orbits, one with $0$ 's only, the other with $1$ 's only,
\item\label{2DEBernoulliSym-w-erg}
for each $k \ge 1$
two countable infinite orbits $\{ (x_i)_{i \in \Z} \in X \mid \sum_{i \in \Z} x_i = k\}$
and $\{ (x_i)_{i \in \Z} \in X \mid \sum_{i \in \Z} (1-x_i) = k\}$,
\item\label{3DEBernoulliSym-w-erg}
and the uncountable orbit, that we denote by $X'$,
of sequences that have infinitely many $0$'s and infinitely many $1$ 's.
\end{enumerate}
In particular, the complement $N$ of $X'$ in $X$ is countable,
the partition $X = X' \sqcup N$ is $\Sy (\Z)$-invariant,
and $\Sy (\Z)$ is transitive on $X'$.
It follows that every $\Sy (\Z)$-invariant subset of $X$ is either inside $N$
or contains $X'$.
\par
Since $N$ is countable and the measure $\mu$ of \ref{3DEPropBruno} without atoms, 
$\mu(N) = 0$.
Hence the action of $\Sy (\Z)$ on $(X, \mu)$ is w-ergodic.
\par
If $k \ge 2$,  the measure $\mu$ is by definition decomposable, i.e.\ not s-ergodic
(Proposition \ref{FominEq}).
\end{proof}

\begin{prop}
% 8.2 
\label{KolpasLC}
Corollary \ref{cor2EDT} does not extend to the situation of the group $\Sy (\Z)$
acting on the compact space $X$.
\par
More precisely, there exists a $\Sy (\Z)$-invariant Borel subset $X'$ of $X$ such that,
for every $p \in ]0,1[$, we have
$\mu_p \in \mathcal E^{\Sy( \Z)}(X)$ and $\mu_p (X') = 1$.
\end{prop}

\begin{proof}
Let $A \subset X$ be a $\Sy (\Z)$-invariant Borel subset.
Suppose that there exists some $p \in ]0,1[$ such that $\mu_p (A) \ne 0$.
In the previous proof,
we have checked that $X \smallsetminus A$ is countable.
Hence, for every $p \in ]0,1[$, we have $\mu_p(A) = 1$.
\end{proof}

Despite the failure of Corollary \ref{cor2EDT} in situations like that of the previous proposition,
we have the following result, that is Theorem 6 of \cite{Fomi--50}:

\begin{prop}[Fomin]
% 8.3 
\label{6deFomin}
Let $G$ be a topological group acting continuously by homeomorphisms
on a compact space $\Omega$.
Assume that there exist two distinct s-ergodic $G$-invariant probability measures on $\Omega$,
say $\nu_1$ and $\nu_2$.
\par
Then there exists a Borel partition of $\Omega$ in two subsets $A_1, A_2$ 
that are $\nu_1$-essentially invariant and $\nu_2$-essentially invariant by $G$,
and such that
\begin{enumerate}[noitemsep]
\item\label{1DEFomin}
$\nu_1(A_1) = 1$ \hskip.2cm and \hskip.2cm $\nu_2(A_1) = 0$,
\item\label{2DEFomin}
$\nu_1(A_2) = 0$ \hskip.2cm and \hskip.2cm $\nu_2(A_2) = 1$.
\end{enumerate}
\end{prop}

When $G$ is a second countable locally compact group,
recall we gave a much stronger conclusion in Corollary \ref{cor2EDT};
in particular, the conclusion of Proposition \ref{6deFomin}
holds with $A_1, A_2$ actually
$G$-invariant, a condition stronger than
the above conditions of essential invariance.

\vskip.2cm

Before giving an illustration of Proposition \ref{6deFomin}
with Kolmogorov's example, 
we recall the following well-known fact.
For $p \in ]0,1[$, let $E_p$ denote the Borel 
subset of $X$  consisting of sequences $(x_i)_{i \in \Z}$
such that
$\lim_{k \to \infty} \frac{1}{2k+1} \sum_{i=-k}^k x_i  = p$,
i.e.\ of sequences in which the $1$ 's have density $p$.
Note that $E_p$
is invariant by $\Sy_f( \Z)$, but not by $\Sy (\Z)$.

\begin{prop}
% 8.4 
\label{Kol+Birk}
Let $p,q \in ]0,1[$ with $p \ne q$, and let $E_p, E_q \subset X$ be as above.
Then $\mu_p (E_p) = 1$ and $\mu_p(E_q) = 0$.
\end{prop}

\begin{proof}
Let $\varphi_0 : X \longrightarrow \R$ be defined by $\varphi_0(x) = x_0$. Then
\begin{equation*}
\int_X \varphi_0 (x) d \mu_p (x) \, = \, p ,
\end{equation*}
by definition of $\mu_p$.
By Birkhoff's ergodicity theorem, the limit
\begin{equation*}
\varphi^*_0(x) \, := \, \lim_{k \to \infty} \frac{1}{2k+1} \sum_{i=-k}^k \varphi_0 (S^i x)
\end{equation*}
exist for $\mu_p$-almost all $x \in X$, and defines a $\mu_p$-almost everywhere constant
function $\varphi^*_0$ of essential value $p$ 
(because the shift $S$ is s-ergodic on $(X, \mu_p)$).
Observe that, for $x \in X$, we have $\varphi^*_0(x) = p$ if and only if $x \in E_p$.
Hence $E_p = X$, up to $\mu_p$-negligible sets;
otherwise said: $\mu_p (E_p) = 1$.
\par

Since $E_p \cap E_q = \emptyset$, we have $\mu_p (E_q) = 0$.
\end{proof}

If we particularize the pair $(G, \Omega)$ 
to the pair $(G = \Sy (\Z), X = \{0,1\}^{\Z})$ of Kolmogorov's example,
and $\nu_1, \nu_2$ to the Bernoulli measures $\mu_{1/3}, \mu_{2/3}$,
then the conclusion of Proposition \ref{6deFomin} holds 
for the subsets $E_{1/3}, E_{2/3}$ of Proposition \ref{Kol+Birk};
for the essential invariance of these two subsets,
see Remark \ref{banalite}.

\vskip.2cm

Let us finally mention a generalization of  Equality (\ref{eqE=E}),
from just before Proposition \ref{PropBruno}, 
and of Part \ref{1DEPropBruno} of the same proposition.
Consider a measure space $(Y, \mathcal C)$
and the product space $(X, \mathcal B)$, with $X = Y^{\Z}$,
with the natural action of the groups $\Sy(\Z)$ and $\Sy_f(\Z)$.
For every probability measure $\lambda$ on $(Y, \mathcal C)$,
let $\widetilde \lambda$ be the probability measure on $(X, \mathcal B)$
that is the product of copies of $\lambda$ indexed by $\Z$;
we denote by $\operatorname{Bern}(X)$ 
the set of measures of the form $\widetilde \lambda$;
observe that $\operatorname{Bern}(X) \subset \mathcal P^{\Sy(\Z)}(X)$.
Then:
\begin{enumerate}[noitemsep]
\item\label{1DEHS}
$\mathcal P^{\Sy(\Z)}(X) = \mathcal P^{\Sy_f(\Z)}(X)$ \cite[Theorem 3.2]{HeSa--63},
and therefore $\mathcal E^{\Sy(\Z)}(X) = \mathcal E^{\Sy_f(\Z)}(X)$
\item\label{2DEHS}
$\mathcal E^{\Sy_f(\Z)}(X) = \operatorname{Bern}(X)$ \cite[Theorem 5.3]{HeSa--63}.
\end{enumerate}

\section{\textbf{Fomin's representations}}
% s9
\label{srepresentations}

In \cite[$\S$~1]{Fomi--50}, Fomin proves the equivalence 
\ref{iDEFominEq} $\Leftrightarrow$ \ref{iiDEFominEq} of Proposition \ref{FominEq}
by adding one more equivalent condition, of independent interest,
in terms of unitary representations.
The object of this section is to describe this  condition.
Interaction between ergodic theory and unitary representations
were pointed out earlier by Koopman \cite{Koop--31}.
\par

Let $G$ be a topological group acting continuously by homeomorphisms
on a compact space $X$.
Let  $\mathcal C (X, \T)$ denote the group of all continuous functions
from $X$ to the compact group $\T$ 
of complex numbers of modulus one.
We consider the natural action of $G$ on $\mathcal C (X, \T)$,
defined by $(g(\varphi))(x) = \varphi(g^{-1}(x))$ 
for all $g \in G$, $\varphi \in \mathcal C (X, \mathcal T)$, and $x \in X$.
In reference to Fomin, 
we denote by $\mathcal F$ the corresponding semi-direct product
$\mathcal C (X, \mathbf T) \rtimes G$ (Fomin's notation is $P$).
We do not furnish the group $\mathcal F$ with any topology.
\par

Consider a $G$-invariant probability measure $\mu$ on $X$,
and the complex Hilbert space $L^2_{\C}(X, \mu)$.
For $(\varphi, g) \in \mathcal F$,
define a linear operator $\rho_\mu(\varphi, g)$ on $L^2_{\C}(X, \mu)$ by
\begin{equation*}
\left( \rho_\mu (\varphi, g) \xi \right) (x) \, = \,
\varphi(x) \xi (g^{-1}(x))
\end{equation*}
for all $\xi \in L^2_{\C}(X, \mu)$ and $x \in X$.
The following proposition, 
which is now straightforward to check, is Theorem 1 in \cite{Fomi--50}.

\begin{prop}
% 9.1
\label{FominTh1}
Let $G, X, \mathcal F$, and $\mu$ be as above.
Then $(\varphi, g) \longmapsto \rho_\mu(\varphi, g)$
defines a unitary representation of the group $\mathcal F$
on the Hilbert space $L^2_{\C}(X, \mu)$.
\end{prop}

For every essentially bounded function $\psi \in L^\infty_{\C}(X, \mu)$,
we denote by $M_\psi$  the multiplication operator $\xi \longmapsto \varphi \xi$
on $L^2_{\C}(X, \mu)$.

\begin{lem}
% 9.2
\label{lemmaforFominTh2}
With the notation above, let $S$ be a continuous linear operator on $L^2_{\C}(X, \mu)$
that commutes with $\rho_\mu(\varphi, g)$ for all $(\varphi, g) \in \mathcal F$.
\par

Then $S = M_\psi$ for some $\psi \in L^\infty_{\C}(X, \mu)$;
moreover $\psi$ is $\mu$-essentially $G$-invariant (in the sense of Definition \ref{defessGinv}).
\end{lem}

\begin{proof}
By hypothesis, $S$ commutes with $\rho_\mu(\varphi, 1) = M_\varphi$ 
for every $\varphi \in \mathcal C(X, \T)$,
and therefore with sums of products of such multiplication operators.
By the Stone-Weierstrass theorem, $S$ commutes also with $M_\varphi$
for every continuous function $\varphi : X \longrightarrow \C$.
By a standard argument, it follows that 
there exists $\psi \in L^\infty_{\C}(X, \mu)$ such that $S = M_\psi$;
see \cite[Chap.\ II, $\S$~3, no 3, Lemma 3]{BourTS}.
\par

Since $S = M_\psi$ commutes with $\rho_\mu(1, g)$ for all $g \in G$,
the function $\psi$ is $\mu$-essentially $G$-invariant.
\end{proof}

The following proposition is Theorem 2 in \cite{Fomi--50}.
Note that Property \ref{iiDEFominTh2} below
coincides with Property \ref{iiDEFominEq} of Proposition \ref{FominEq}.

\begin{prop}
% 9.3
\label{FominTh2}
Let $G, X, \mathcal F$, and $\mu$ be as above.
The following properties are equivalent:
\begin{enumerate}[noitemsep,label=(\roman*)]
\item\label{iDEFominTh2}
the representation $\rho_\mu$ is irreducible,
\item\label{iiDEFominTh2}
the measure $\mu$ is s-ergodic.
\end{enumerate}
\end{prop}

\begin{proof}
For \ref{iDEFominTh2} $\Rightarrow$ \ref{iiDEFominTh2},
we prove the contraposition.
If $\mu$ is not s-ergodic, 
there exists a $\mu$-essentially invariant Borel subset $A \subset X$
with $0 < \mu(A) < 1$.
The subspace of $L^2_{\C}(X, \mu)$ of functions which vanish outside $A$
is invariant by $\rho_\mu$, and therefore the representation is reducible.
\par

The converse implication
\ref{iiDEFominTh2} $\Rightarrow$ \ref{iDEFominTh2}
follows from Lemma \ref{lemmaforFominTh2} and Schur's lemma
(for which we refer to \cite[Theorem A.2.2]{BeHV--08}).
\end{proof}

\end{document}